\documentclass[11pt]{article}
\usepackage[a4paper,margin=1in]{geometry}
\usepackage{amsmath,amssymb,amsfonts,mathrsfs,mathtools}
\usepackage{amsthm}
\usepackage{xcolor}
\usepackage[colorlinks=true,citecolor=blue,linkcolor=blue,urlcolor=blue]{hyperref}
\usepackage{enumitem}
\allowdisplaybreaks
\newtheorem{theorem}{Theorem}[section]
\newtheorem{proposition}[theorem]{Proposition}
\newtheorem{lemma}[theorem]{Lemma}
\newtheorem{corollary}[theorem]{Corollary}
\theoremstyle{definition}

\theoremstyle{remark}

\numberwithin{equation}{section}

\newcommand{\R}{\mathbb R}
\newcommand{\Z}{\mathbb Z}
\newcommand{\cB}{\mathcal B}
\newcommand{\cF}{\mathcal F}
\newcommand{\cS}{\mathcal S}
\newcommand{\supp}{\operatorname{supp}}
\newcommand{\esssup}{\operatorname*{ess\,sup}}
\newcommand{\essinf}{\operatorname*{ess\,inf}}

\title{Time-refined Triebel--Lizorkin Estimates and Applications to Keller--Segel Type Equations}
\author{Wenhai Shan$^{a,}$\thanks{E-mail: mathshanwenhai@163.com}\quad
Xiao-Song Yang$^{a,b,}$\thanks{Corresponding author: yangxs@hust.edu.cn}\\
\small $^{a}$School of Mathematics and Statistics, Huazhong University of Science and Technology,\\
\small Wuhan, Hubei 430074, P.R. China\\
\small $^{b}$Hubei Key Laboratory of Engineering Modeling and Scientific Computing,\\
\small Huazhong University of Science and Technology, Wuhan, Hubei 430074, P.R. China}
\date{}

\begin{document}
\maketitle

\begin{abstract}
We develop heat-flow estimates in critical homogeneous Triebel--Lizorkin
spaces and apply them to two-dimensional Keller--Segel type equations.
For $q>2$, we show that the heat-flow from
$\dot F^0_{1,q}(\mathbb R^2)$ to
$L^2(0,T;\dot F^1_{1,q}(\mathbb R^2))$ is unbounded.  This motivates the introduction of new spaces, time-refined Triebel--Lizorkin spaces, in which the time norm is taken
before the dyadic summation.  In this framework, we establish a family
of heat smoothing estimates together with the corresponding maximal regularity.  We further prove Banach-valued Peetre
estimates, Banach-valued Jawerth-type estimate, homogeneous Poisson
estimate, and an endpoint bilinear estimate for the Keller--Segel drift.
As an application, we obtain local well-posedness for arbitrary initial
data and global well-posedness for sufficiently small initial data in the
critical space $\dot F^0_{1,q}(\mathbb R^2)$, $2\le q<\infty$, for the
two-dimensional parabolic--elliptic Keller--Segel equation.  The same
analytic framework also yields analogous well-posedness results for a
related two-component drift--diffusion system.
\end{abstract}



\section{Introduction}

\subsection{Heat flow, critical spaces, and time-refined regularity}

The heat equation is a fundamental model witnessing the interaction between semigroup theory,
harmonic analysis, and nonlinear parabolic partial differential equations.  For the
inhomogeneous problem
\[
\partial_t u-\Delta u=F,
\qquad
u|_{t=0}=u_0,
\]
the corresponding mild solution is given by
\[
u(t)=e^{t\Delta}u_0+\int_0^t e^{(t-\tau)\Delta}F(\tau)\,d\tau.
\]
The mapping properties of the free heat semigroup and the Duhamel
operator form the basic linear framework for fixed-point arguments in
nonlinear parabolic equations {\color{blue}and often play a decisive role in determining the range of well-posedness; see \cite{Weissler1981}. 
	
At the Fourier level, the decay of the multiplier
\(e^{-t|\xi|^2}\) suppresses high frequencies and produces spatial
smoothing. Heat estimates quantify the decay and smoothing effects of the heat
	flow.  They determine the natural resolution spaces, control both the
	free evolution and the response to inhomogeneous forcing, and provide
	the short-time smallness required to close contraction
	arguments. Combined with suitable nonlinear estimates, they provide the basic
	linear tools for constructing solutions, proving uniqueness and
	continuous dependence, and establishing parabolic regularization;
	see, for example, \cite{Henry1981,Pazy1983,Amann1995}.
	
	More generally, heat estimates
	describe how temporal integrability of the forcing can be converted
	into spatial regularity of the solution. For the Duhamel operator, this principle leads naturally to a family of estimates mapping a forcing term with time integrability \(L^a\) to a
solution with time integrability \(L^r\), where \(a\le r\).  The gain of
spatial regularity is determined by the parabolic time scale.  In the
diagonal case \(a=r\), one recovers maximal regularity: the time
derivative and the highest-order spatial derivative belong to the same
Bochner space as the forcing term.  When \(a<r\), one obtains off-diagonal
estimates, in which part of the spatial smoothing is exchanged
for improved time integrability.  Such estimates are especially useful
when the nonlinear forcing and the desired solution space carry
different natural time exponents; see
\cite{Amann1995,DoreVenni1987,Weis2001}}

{\color{blue}
	To formulate the heat estimates discussed above, one has to choose a
	spatial Banach space that reflects both the parabolic scaling of the
	equation and the structure of the nonlinear terms.  Classical
	heat estimates are commonly developed in Lebesgue and
	Sobolev--Bessel potential spaces; see
	\cite{Pazy1983,Amann1995}.  At a more abstract level, maximal
	\(L^\rho\)-regularity is studied for evolution equations on Banach
	spaces, with the UMD property and suitable \(R\)-boundedness assumptions
	playing a central role in the standard theory; see
	\cite{DoreVenni1987,Weis2001}.
	
In the analysis of nonlinear equations at scaling-critical regularity,
Besov and Triebel--Lizorkin spaces constitute two principal
Littlewood--Paley scales.  Both are well adapted to parabolic equations,
since the dyadic decomposition makes the smoothing action of the heat
semigroup and the frequency interactions of the nonlinear terms
explicit.  Their different norm structures, however, lead to distinct
analytic features; see
\cite{Triebel1983,BahouriCheminDanchin2011}.

Heat estimates in the Besov scale have been extensively developed and
play a central role in the study of nonlinear parabolic equations.  The endpoint case with spatial \(L^1\)-integrability is substantially more delicate, since it lies outside the
standard UMD-based maximal-regularity theory.  In
\cite{OgawaShimizu2010}, Ogawa and Shimizu established endpoint maximal regularity for the heat equation in the critical Besov spaces
\[
\dot B^0_{1,\rho}(\mathbb R^n),
\qquad
1<\rho\leq\infty.
\]
They further applied the case \(\rho=2\) to the two-dimensional
Keller--Segel system in
\(\dot B^0_{1,2}(\mathbb R^2)\).

These results naturally raise the corresponding endpoint question in
the Triebel--Lizorkin space \(\dot F^s_{1,q}\).  Let
\((\dot\Delta_j)_{j\in\mathbb Z}\) be a homogeneous
Littlewood--Paley decomposition.  For the parameters considered below,
we define
\[
\|f\|_{\dot F^s_{1,q}}
=
\left\|
\left(
2^{js}|\dot\Delta_j f(x)|
\right)_{j\in\mathbb Z}
\right\|_{L^1_x(\ell^q_j)}.
\]
The essential distinction from the corresponding Besov norm is that
the dyadic \(\ell^q\)-summation is performed pointwise in the spatial
variable before the final \(L^1_x\)-integration; see
\cite{Triebel1983,FrazierJawerth1985}.  This ordering becomes especially
significant after a time norm is introduced, since time integration,
spatial integration, and dyadic summation cannot in general be
interchanged without changing the resulting norm.  Accordingly, for
	the critical initial space \(\dot F^0_{1,q}\), the natural first candidate for capturing the gain of one spatial
	derivative in \(L^2_t\) is the standard
Bochner space}
\[
L^2(0,T;\dot F^1_{1,q}).
\]
In this norm, the full spatial Triebel--Lizorkin norm is formed at each
fixed time before the temporal $L^2$-integration is taken.  We first show
that this ordering is incompatible with the dyadic heat scales whenever
$q>2$.

\begin{proposition}[Obstruction of the standard Bochner heat estimate]
\label{prop:failure-ordinary-heat}
Let $q>2$ and $T>0$. There is no constant $C>0$ satisfying
\[
\|e^{t\Delta}u_0\|_{L^2(0,T;\dot F^1_{1,q})}
\le
C\|u_0\|_{\dot F^0_{1,q}}
\]
for every $u_0\in\dot F^0_{1,q}(\mathbb R^2)$.
\end{proposition}

{\color{blue}The obstruction in Proposition~\ref{prop:failure-ordinary-heat} comes from
	combining dyadic blocks that are active on different heat time scales.
	For initial data consisting of \(N\) well-separated blocks, the
	\(\dot F^0_{1,q}\)-norm is of size \(N^{1/q}\), whereas their essentially
	disjoint contributions in time are accumulated by the \(L^2_t\)-norm
	with size \(N^{1/2}\).  Thus, the claimed estimate would imply
	\[
	N^{1/2}\lesssim N^{1/q},
	\]
	which is impossible for \(q>2\).  Each block still enjoys the expected
	heat smoothing, the obstruction is caused only by the order of time
	integration and frequency summation in the standard Bochner norm.  This
	also explains why \(q=2\) is the distinguished case.

To place the time norm at the level of each dyadic block, motivated by
 the Chemin--Lerner refinement of Besov spaces
\cite{CheminLerner1995}, we introduce a new norm, the time-refined Triebel–Lizorkin norm
\[
\|u\|_{\dot F^s_{1,q}(L^\rho_T)}
:=
\left\|
\left(
2^{js}
\|\dot\Delta_j u(\cdot,x)\|_{L^\rho(0,T)}
\right)_{j\in\mathbb Z}
\right\|_{L^1_x\ell^q_j}.
\]}
 Equivalently,
this norm may be viewed as the Banach-valued Triebel--Lizorkin norm
\[
\dot F^s_{1,q}
\bigl(\mathbb R^2;L^\rho(0,T)\bigr).
\]
We refer to this new space as a \emph{time-refined Triebel--Lizorkin space}.
The word ``refined'' describes the time--frequency ordering and does not
assert a universal inclusion relation with the standard Bochner space.

{\color{blue}
The advantage of this time-refined space is that the time norm is taken
separately on each dyadic block, in accordance with its natural heat
time scale, before the dyadic summation and spatial integration are
performed.  This preserves the pointwise frequency structure of the
Triebel--Lizorkin norm and recovers the heat estimates that fail in the
standard Bochner setting. More importantly, it provides a common framework for the free heat flow and the Duhamel operator.
}

The Banach-valued formulation reveals a second difficulty, now associated
with the spatial endpoint $p=1$.  The usual vector-valued
Hardy--Littlewood maximal inequality cannot be applied directly at the
$L^1_x$ level.  A useful precedent is the work of Guo and Li
\cite{GuoLi2021} on the Euler equations in critical Triebel--Lizorkin
spaces.  A key ingredient in their treatment of the endpoint case is the
use of Peetre-type maximal estimates \cite{Peetre1975} for frequency-localized functions.
By taking a sufficiently small auxiliary power before invoking the
Fefferman--Stein vector-valued maximal inequality \cite{FeffermanStein1971}, this mechanism bypasses
the failure of the Hardy--Littlewood maximal operator to be bounded on
$L^1$. 

We adapt this endpoint mechanism to the present Banach-valued setting.
More precisely, for functions taking values in $L^\rho(0,T)$, we establish a Banach-valued Peetre estimate (Lemma \ref{lem:peetre-pointwise}) and derive the corresponding
sequence-valued estimate in $L^1_x\ell^q$ (Lemma \ref{lem:almost}).  This allows us to lift
frequency-localized heat-kernel estimates to the full time-refined
Triebel--Lizorkin norm. The first main result is stated as follows.

\begin{theorem}[Generalized time-refined heat estimate]
\label{thm:general-heat}
Let \(1\le q<\infty\), \(1\le a\le r\le\infty\), and \(s\in\mathbb R\).  Set
\[
\sigma
=
2\left(1+\frac1r-\frac1a\right).
\]
Let
\[
u(t)
=
e^{t\Delta}u_0
+
\int_0^t e^{(t-\tau)\Delta}F(\tau)\,d\tau
\]
solve
\[
\partial_tu-\Delta u=F,
\qquad
u|_{t=0}=u_0.
\]
If
\[
u_0\in\dot F^{s+2-2/a}_{1,q},
\qquad
F\in\dot F^s_{1,q}(L^a_T),
\]
then
\[
\|u\|_{\dot F^{s+\sigma}_{1,q}(L^r_T)}
\lesssim
\|u_0\|_{\dot F^{s+2-2/a}_{1,q}}
+
\|F\|_{\dot F^s_{1,q}(L^a_T)}.
\]
Moreover,
\[
\|\Delta u\|_{\dot F^{s+\sigma-2}_{1,q}(L^r_T)}
\lesssim
\|u_0\|_{\dot F^{s+2-2/a}_{1,q}}
+
\|F\|_{\dot F^s_{1,q}(L^a_T)},
\]
and
\[
\partial_tu
\in
\dot F^{s+\sigma-2}_{1,q}(L^r_T)
+
\dot F^s_{1,q}(L^a_T).
\]
All implicit constants are independent of \(T\).
\end{theorem}

Theorem~\ref{thm:general-heat} contains both the off-diagonal estimates needed for
the nonlinear application and the diagonal maximal-regularity estimate.  In
particular, the choices
\[
(a,r,s)=(1,\infty,0)
\qquad\text{and}\qquad
(a,r,s)=(1,2,0)
\]
give, respectively,
\[
\left\|
\int_0^t e^{(t-\tau)\Delta}F(\tau)\,d\tau
\right\|_{\dot F^0_{1,q}(L^\infty_T)}
\lesssim
\|F\|_{\dot F^0_{1,q}(L^1_T)}
\]
and
\[
\left\|
\int_0^t e^{(t-\tau)\Delta}F(\tau)\,d\tau
\right\|_{\dot F^1_{1,q}(L^2_T)}
\lesssim
\|F\|_{\dot F^0_{1,q}(L^1_T)}.
\]
{\color{blue}
These off-diagonal estimates provide the bounds needed to control the
Keller--Segel nonlinearity. The former controls the Duhamel term uniformly in the critical norm,
while the latter provides the spatial smoothing and time integrability
needed to handle the quadratic nonlinearity. Together, these bounds allow the Duhamel
mapping to be closed in the solution space.} When \(a=r=\rho\), one has \(\sigma=2\), and the theorem yields the corresponding
time-refined maximal-regularity estimate. 

\subsection{The Keller--Segel equation and critical endpoint regularity}

We apply the preceding linear theory to the two-dimensional parabolic--elliptic
Keller--Segel equation
\begin{equation}
\label{eq:KS-intro}
\partial_tu-\Delta u
+
\nabla\cdot\left(u\nabla(-\Delta)^{-1}u\right)
=0,
\qquad
u|_{t=0}=u_0.
\end{equation}
The Keller--Segel model was introduced in \cite{KellerSegel1970} to
describe chemotactic aggregation and was later studied from a nonlinear
and collapse-oriented viewpoint in \cite{ChildressPercus1981}.  Its analysis reflects the competition
between diffusion and the attractive drift generated by the density itself.  The
classical theory includes local and global existence, critical-mass phenomena,
concentration, asymptotic behavior, and finite-time blow-up; see
\cite{Biler1998,Nagai1995,HerreroVelazquez1997,Horstmann2003,
BlanchetDolbeaultPerthame2006,BlanchetCarrilloMasmoudi2008,CalvezCorrias2008}
and the references therein.

{\color{blue}Equation~\eqref{eq:KS-intro} is invariant under the parabolic scaling
\[
u_\lambda(t,x)=\lambda^2u(\lambda^2t,\lambda x),
\qquad
u_{0,\lambda}(x)=\lambda^2u_0(\lambda x).
\]
This scaling naturally motivates the study of the Cauchy problem in
critical function spaces.  For the parabolic--parabolic system,
small-data global existence in critical spaces, global existence below
the critical mass, and mild solutions in scaling-invariant spaces were
studied in
\cite{CorriasPerthame2006,CalvezCorrias2008,KozonoSugiyamaWachi2012}.
For the parabolic--elliptic system, local strong solutions for initial
data in weak \(L^{n/2}\), together with small-data global existence and
blow-up-rate estimates, were established in
\cite{KozonoSugiyama2010}. In the homogeneous
Besov setting, Iwabuchi \cite{Iwabuchi2011} established small-data
global well-posedness together with an ill-posedness result, while Zhao
\cite{Zhao2018} proved local and small-data global well-posedness, as
well as Gevrey analyticity, for generalized Keller--Segel equations in
a broad range of critical Besov spaces.  Critical Besov--Morrey results
were later obtained by Nogayama and Sawano
\cite{NogayamaSawano2022}.  At the endpoint \(p=1\), Ogawa and Shimizu
\cite{OgawaShimizu2008} established local well-posedness for large data
and global well-posedness for small data for a two-dimensional
drift--diffusion system in the critical Hardy space
\(\mathcal H^1(\mathbb R^2)\).

For a homogeneous Triebel--Lizorkin norm in dimension two, the scaling-critical regularity at the spatial endpoint \(p=1\) is
\[
\dot F^0_{1,q}(\mathbb R^2),
\qquad
2\le q<\infty.
\]
For \(q=2\), this space coincides with
the Hardy space \(\mathcal H^1(\mathbb R^2)\). Therefore, what is not covered directly by these results is the larger endpoint Triebel--Lizorkin scale \(\dot F^0_{1,q}\) with \(q>2\), where the pointwise dyadic
summation is no longer compatible with the standard Bochner smoothing norm. This is
the gap addressed here.}

For \(T\in(0,\infty]\), define
\[
X_T^q
:=
\dot F^0_{1,q}(L^\infty_T)
\cap
\dot F^1_{1,q}(L^2_T).
\]
The choice of \(L^2_t\) is dictated by the quadratic structure of the drift: the heat
flow gains one derivative in \(L^2_t\), while Hölder's inequality gives
\[
L^2_t\times L^2_t\longrightarrow L^1_t.
\]
Thus the nonlinear term should be estimated in \(\dot F^0_{1,q}(L^1_T)\), precisely
the forcing space covered by the two off-diagonal consequences of
Theorem~\ref{thm:general-heat}.

The key nonlinear estimate is the following.

\begin{proposition}[Keller--Segel bilinear estimate]
\label{prop:bilinear}
Let \(2\le q<\infty\).  Then
\[
\left\|
\nabla\cdot\left(f\nabla(-\Delta)^{-1}g\right)
\right\|_{\dot F^0_{1,q}(L^1_T)}
\lesssim
\|f\|_{\dot F^1_{1,q}(L^2_T)}
\|g\|_{\dot F^1_{1,q}(L^2_T)}.
\]
Consequently,
\[
\left\|
\int_0^t e^{(t-\tau)\Delta}
\nabla\cdot\left(f\nabla(-\Delta)^{-1}g\right)(\tau)\,d\tau
\right\|_{X_T^q}
\lesssim
\|f\|_{\dot F^1_{1,q}(L^2_T)}
\|g\|_{\dot F^1_{1,q}(L^2_T)}.
\]
\end{proposition}
{\color{blue}This proposition shows that the Keller--Segel nonlinearity is closed
	within the time-refined critical framework: the quadratic drift belongs
	to the natural forcing space, and its Duhamel contribution lies in
	\(X_T^q\).  Together with the linear heat estimates, this provides the
	quadratic control required for the fixed-point argument and yields the
	following main theorem.} 

\begin{theorem}[Critical well-posedness for the Keller--Segel equation]
\label{thm:ks-classical}
Let \(2\le q<\infty\) and \(u_0\in\dot F^0_{1,q}(\mathbb R^2)\).  Then there exists
\(T>0\) such that equation~\eqref{eq:KS-intro} has a unique mild solution in the full
class
\[
u\in X_T^q.
\]
Moreover,
\[
u\in C\bigl([0,T];\dot F^0_{1,q}(\mathbb R^2)\bigr),
\]
and the data-to-solution map is continuous from a neighborhood of \(u_0\) in
\(\dot F^0_{1,q}\) into \(X_T^q\) on a common sufficiently small time interval.  There
exists \(\varepsilon_q>0\) such that, if
\[
\|u_0\|_{\dot F^0_{1,q}}\le\varepsilon_q,
\]
then the solution is global and belongs to \(X_\infty^q\).
\end{theorem}

The same analytic framework applies to the two-component drift--diffusion system
\begin{equation}
\label{eq:drift-system-intro}
\left\{
\begin{aligned}
\partial_t v-\Delta v+\nabla\cdot(w\nabla\psi)&=0,\\
\partial_t w-\Delta w+\lambda w+\nabla\cdot(v\nabla\psi)&=0,\\
-\Delta\psi&=w,
\end{aligned}
\right.
\qquad
\lambda\ge0.
\end{equation}
Indeed, both nonlinearities are covered by Proposition~\ref{prop:bilinear}, and the
factor \(e^{-\lambda t}\) only improves the linear estimates.

\begin{corollary}[The drift--diffusion system]
\label{cor:drift-diffusion}
Let \(2\le q<\infty\) and
\[
(v_0,w_0)
\in
\dot F^0_{1,q}(\mathbb R^2)
\times
\dot F^0_{1,q}(\mathbb R^2).
\]
Then system~\eqref{eq:drift-system-intro} has a unique local mild solution
\[
(v,w)\in X_T^q\times X_T^q
\]
for some \(T>0\).  If
\[
\|v_0\|_{\dot F^0_{1,q}}
+
\|w_0\|_{\dot F^0_{1,q}}
\]
is sufficiently small, then the solution is global.
\end{corollary}

\subsection{Main proof ideas and organization}

The linear argument begins with a frequency-localized heat-kernel estimate.  If
\(K_j\) denotes the heat kernel restricted to an annulus of size \(2^j\), then, for
every sufficiently large integer \(N\),
\[
|K_j(t,x)|
\lesssim_N
2^{2j}e^{-ct2^{2j}}(1+2^j|x|)^{-N}.
\]
The temporal decay produces the factor \(2^{-2j/b}\) in \(L^b_t\), while the spatial
convolution is controlled by the Banach-valued Peetre maximal estimate.  Young's
inequality in time then gives the off-diagonal gain in
Theorem~\ref{thm:general-heat}; the diagonal case yields maximal regularity.

For the nonlinear estimate, we use Bony's decomposition \cite{Bony1981}
\[
fb=T_bf+T_fb+R(f,b),
\qquad
b=\nabla(-\Delta)^{-1}g.
\]
The low--high term is treated pointwise in space by the Banach-valued Peetre estimate.
The reverse interaction and the comparable-frequency remainder are placed in
\(\dot B^0_{1,1}(L^1_T)\), using the Banach-valued Jawerth-type estimate to obtain
summability of the dyadic \(L^2_TL^2_x\)-blocks.  A dyadic order \(-1\) multiplier
argument gives the Poisson gain and defines the homogeneous velocity as an absolutely
convergent series in \(L^\infty_xL^2_T\).  These estimates close the Duhamel map in
\(X_T^q\).  Small-time smoothing of the free heat flow gives local existence for
arbitrary data, while time localization yields uniqueness in the full solution class.
The constants are independent of \(T\), which permits the global fixed-point argument
for sufficiently small initial data.

The remainder of the paper is organized as follows.  Section~2 introduces the
time-refined Triebel--Lizorkin spaces and establishes the Banach-valued Peetre and
frequency-summation estimates used at the endpoint.  Section~3 proves
Proposition~\ref{prop:failure-ordinary-heat}, establishes the localized heat-kernel
bounds, and proves Theorem~\ref{thm:general-heat} together with its
maximal-regularity consequence.  Section~4 develops the Banach-valued Bernstein and
multiplier estimates, the Jawerth-type estimate, and the
homogeneous Poisson reconstruction.  Section~5 proves
Proposition~\ref{prop:bilinear}, establishes the continuity and small-time smoothing
properties, and proves Theorem~\ref{thm:ks-classical}.  Finally, Section~6 proves
Corollary~\ref{cor:drift-diffusion}.

\section{Preliminaries and endpoint Peetre estimates}
\label{sec:prelim}

This section collects the harmonic-analysis tools that will be used
throughout the paper. Let $(\dot\Delta_j)_{j\in\Z}$ be a homogeneous Littlewood--Paley
decomposition on $\R^2$, and let
\[
\dot S_j f=\sum_{k<j}\dot\Delta_kf.
\]
All homogeneous spaces are understood in the usual quotient $\cS'(\R^2)/\mathcal P$
modulo polynomials.  For $1\le q<\infty$ and $s\in\R$, set
\[
\|f\|_{\dot F^s_{1,q}}
=
\left\|
\left(2^{js}|\dot\Delta_jf|\right)_{j\in\Z}
\right\|_{L^1_x\ell^q_j}.
\]
For an interval $I\subset\R$ and $1\le\rho\le\infty$, define
\[
\|u\|_{\dot F^s_{1,q}(L^\rho(I))}
=
\left\|
\left(2^{js}\|\dot\Delta_ju(\cdot,x)\|_{L^\rho(I)}\right)_{j\in\Z}
\right\|_{L^1_x\ell^q_j}.
\]
When $I=(0,T)$ we write $L_T^\rho$.  We use analogous notation for
$\dot B^s_{p,q}(L^\rho(I))$.

The first estimate is the Banach-valued version of the Peetre maximal
inequality.  Its proof uses scalarization rather than applying the scalar
inequality to $\|u\|_E$, which need not be frequency localized. 

\begin{lemma}[Banach-valued Peetre estimate]
\label{lem:peetre-pointwise}
Let $E$ be a Banach space, let $0<\theta<\infty$, and let $M>d/\theta$.
If $u:\R^d\to E$ has Fourier support in a fixed compact set, then
\[
\sup_{y\in\R^d}
\frac{\|u(x-y)\|_E}{(1+|y|)^M}
\lesssim
\bigl[\mathcal M(\|u\|_E^\theta)(x)\bigr]^{1/\theta}.
\]
Consequently, if
\[
\supp\widehat u_j\subset\{\xi:|\xi|\le C_0 2^j\},
\]
then, uniformly in $j$,
\[
\sup_{y\in\R^d}
\frac{\|u_j(x-y)\|_E}{(1+2^j|y|)^M}
\lesssim
\bigl[\mathcal M(\|u_j\|_E^\theta)(x)\bigr]^{1/\theta}.
\]
\end{lemma}

\begin{proof}
We first prove the assertion at a fixed frequency scale.  Let
$\ell\in E^*$ with $\|\ell\|_{E^*}\le1$ and define the scalar function
\[
u_\ell(x)=\ell(u(x)).
\]
Since $\ell$ is continuous and linear, scalarization does not enlarge Fourier
support:
\[
\supp\widehat{u_\ell}\subset\supp\widehat u.
\]
The scalar Peetre maximal inequality (see Lemma 2.3 in \cite{GuoLi2021}), with the same compact frequency set,
therefore gives
\begin{equation}
\label{eq:scalar-peetre}
\sup_{y\in\R^d}
\frac{|u_\ell(x-y)|}{(1+|y|)^M}
\lesssim
\bigl[\mathcal M(|u_\ell|^\theta)(x)\bigr]^{1/\theta}.
\end{equation}
The implicit constant is uniform for $\ell$ in the unit ball of $E^*$.  Since
\[
|u_\ell(z)|\le \|u(z)\|_E,
\]
monotonicity of the Hardy--Littlewood maximal operator yields
\[
\mathcal M(|u_\ell|^\theta)(x)
\le
\mathcal M(\|u\|_E^\theta)(x).
\]
Combining \eqref{eq:scalar-peetre} with the preceding maximal-function
comparison, we obtain
\begin{equation}
\label{eq:scalarized-peetre-majorant}
\sup_y
\frac{|\ell(u(x-y))|}{(1+|y|)^M}
\lesssim
\bigl[\mathcal M(\|u\|_E^\theta)(x)\bigr]^{1/\theta}.
\end{equation}
Taking the supremum in \eqref{eq:scalarized-peetre-majorant} over all
$\ell\in E^*$ with $\|\ell\|_{E^*}\le1$ and using the Hahn--Banach identity
\[
\|v\|_E=\sup_{\|\ell\|_{E^*}\le1}|\ell(v)|
\]
gives
\begin{equation}
\label{eq:banach-peetre-fixed}
\sup_y
\frac{\|u(x-y)\|_E}{(1+|y|)^M}
\lesssim
\bigl[\mathcal M(\|u\|_E^\theta)(x)\bigr]^{1/\theta}.
\end{equation}

We now pass to the dyadic form.  Set
\[
v_j(X)=u_j(2^{-j}X).
\]
Then $\widehat v_j$ is supported in a fixed ball independent of $j$.  Applying
the fixed-scale estimate to $v_j$ at $X=2^jx$ gives
\[
\sup_{Z\in\R^d}
\frac{\|v_j(2^jx-Z)\|_E}{(1+|Z|)^M}
\lesssim
\bigl[\mathcal M(\|v_j\|_E^\theta)(2^jx)\bigr]^{1/\theta}.
\]
With $Z=2^jy$, the left-hand side becomes the desired dyadic Peetre
maximal function.  Moreover, the maximal operator respects this scaling:
\begin{equation}
\label{eq:peetre-maximal-scaling}
\mathcal M(\|v_j\|_E^\theta)(2^jx)
=
\mathcal M(\|u_j\|_E^\theta)(x).
\end{equation}
Applying \eqref{eq:banach-peetre-fixed} to $v_j$ and then using
\eqref{eq:peetre-maximal-scaling} proves the second assertion uniformly in
$j$.
\end{proof}

Fix henceforth
\[
0<\theta<\min\{1,q\},
\qquad
M>\frac{2}{\theta}.
\]
For an $E$-valued function localized at scale $2^j$, define
\[
u_j^*(x):=
\sup_{y\in\R^2}
\frac{\|u_j(x-y)\|_E}{(1+2^j|y|)^M}.
\]
The pointwise estimate above becomes useful at the spatial endpoint only
after the dyadic family is summed.  Taking a small auxiliary power allows us
to apply the Fefferman--Stein inequality at exponents strictly larger than
one, which gives the following sequence estimate.

\begin{corollary} 
\label{cor:peetre-sequence}
Let $1\le q<\infty$ and let $E$ be a Banach space.  Suppose
$\supp\widehat u_j\subset\{|\xi|\le C2^j\}$.  Then
\[
\|(u_j^*)_j\|_{L^1_x\ell^q_j}
\lesssim
\|\bigl(\|u_j(x)\|_E\bigr)_j\|_{L^1_x\ell^q_j}.
\]
\end{corollary}

\begin{proof}
By Lemma~\ref{lem:peetre-pointwise},
\begin{equation}
\label{eq:peetre-power-pointwise}
(u_j^*(x))^\theta
\lesssim
\mathcal M(\|u_j\|_E^\theta)(x).
\end{equation}
Using the elementary identity between mixed norms after taking the
$\theta$-th power, \eqref{eq:peetre-power-pointwise} yields
\begin{align}
\label{eq:peetre-powered-mixed}
\|(u_j^*)_j\|_{L^1_x\ell^q_j}^{\theta}
&=
\|((u_j^*)^\theta)_j\|_{L^{1/\theta}_x\ell^{q/\theta}_j}\\
&\lesssim
\|(\mathcal M(\|u_j\|_E^\theta))_j
  \|_{L^{1/\theta}_x\ell^{q/\theta}_j}.
\end{align}
Our choice $0<\theta<\min\{1,q\}$ guarantees
\[
\frac1\theta>1,
\qquad
\frac q\theta>1.
\]
The Fefferman--Stein vector-valued maximal inequality is therefore
applicable and yields
\begin{equation}
\label{eq:peetre-fefferman-stein}
\|(\mathcal M(\|u_j\|_E^\theta))_j
  \|_{L^{1/\theta}_x\ell^{q/\theta}_j}
\lesssim
\|(\|u_j\|_E^\theta)_j
  \|_{L^{1/\theta}_x\ell^{q/\theta}_j}.
\end{equation}
Substituting \eqref{eq:peetre-fefferman-stein} into
\eqref{eq:peetre-powered-mixed}, the last expression equals
\[
\|\bigl(\|u_j(x)\|_E\bigr)_j\|_{L^1_x\ell^q_j}^{\theta}.
\]
Taking the $\theta$-th root completes the proof.
\end{proof}

We next combine the sequence-valued Peetre bound with the finite overlap of
annular Fourier supports.  The resulting lemma is a convenient reconstruction
estimate for dyadic sums and will be used in the low--high paraproduct
analysis.

\begin{lemma}
\label{lem:almost}
Let $1\le q<\infty$, let $E$ be a Banach space, and suppose
\[
\supp\widehat u_k
\subset
\{\xi:c2^k\le|\xi|\le C2^k\}.
\]
Then  
\[
\left\|\sum_k u_k\right\|_{\dot F^0_{1,q}(\R^2;E)}
\lesssim
\|(u_k^*)_k\|_{L^1_x\ell^q_k}
\lesssim
\|\bigl(\|u_k(x)\|_E\bigr)_k\|_{L^1_x\ell^q_k}.
\]
\end{lemma}

\begin{proof}
  Since the Fourier supports of $u_k$ are
contained in annuli of size $2^k$, there is an integer $N_0$, depending only
on the Littlewood--Paley decomposition and on $c,C$, such that
\begin{equation}
\label{eq:finite-frequency-overlap}
\dot\Delta_j u_k=0
\qquad\text{whenever }|j-k|>N_0.
\end{equation}
Consequently,
\begin{equation}
\label{eq:finite-frequency-reconstruction}
\dot\Delta_j\Bigl(\sum_k u_k\Bigr)
=
\sum_{|k-j|\le N_0}\dot\Delta_j u_k.
\end{equation}
Let $K_j$ denote the convolution kernel of $\dot\Delta_j$.  By definition of
$u_k^*$,
\[
\|u_k(x-y)\|_E
\le
u_k^*(x)(1+2^k|y|)^M.
\]
Therefore, for $|j-k|\le N_0$,
\[
\begin{aligned}
\|\dot\Delta_j u_k(x)\|_E
&\le
\int_{\R^2}|K_j(y)|\,\|u_k(x-y)\|_E\,dy\\
&\le
u_k^*(x)
\int_{\R^2}|K_j(y)|(1+2^k|y|)^M\,dy.
\end{aligned}
\]
Since $2^j\simeq2^k$ in this range and
$K_j(y)=2^{2j}K(2^jy)$ with $K\in\cS$, the weighted integral is bounded
uniformly in $j$ and $k$.  Thus
\begin{equation}
\label{eq:localized-block-peetre}
\|\dot\Delta_j u_k(x)\|_E\lesssim u_k^*(x),
\qquad |j-k|\le N_0.
\end{equation}
Combining \eqref{eq:finite-frequency-reconstruction} and
\eqref{eq:localized-block-peetre}, we find
\begin{equation}
\label{eq:finite-convolution-pointwise}
\left(\sum_j
\left\|\dot\Delta_j\sum_k u_k(x)\right\|_E^q
\right)^{1/q}
\lesssim
\left(
\sum_j
\left(\sum_{|k-j|\le N_0}u_k^*(x)\right)^q
\right)^{1/q}.
\end{equation}
The finite convolution in the sequence variable is bounded on $\ell^q$.
Integrating \eqref{eq:finite-convolution-pointwise} in $x$ therefore gives
\begin{equation}
\label{eq:finite-synthesis-estimate}
\left\|\sum_k u_k\right\|_{\dot F^0_{1,q}(E)}
\lesssim
\|(u_k^*)_k\|_{L^1_x\ell^q_k}.
\end{equation}
The second inequality in the statement follows from
Corollary~\ref{cor:peetre-sequence}.

\end{proof}

\section{Heat-flow estimates}
\label{sec:heat}

This section establishes the linear estimates underlying the nonlinear
application.  We first prove that the standard Bochner smoothing estimate
fails for $q>2$.  We then derive a localized heat-kernel bound and combine it
with the Peetre estimates from Section~\ref{sec:prelim} to prove the
time-refined smoothing theorem and its diagonal maximal-regularity
consequence.

\subsection{Proof of the ordinary-space obstruction}

We begin with the counterexample announced in the Introduction.  Its purpose
is to isolate the incompatibility between the dyadic heat time scales and the
order of summation in the standard Bochner norm.

\begin{proof}[Proof of Proposition~\ref{prop:failure-ordinary-heat}]
Choose a nonzero $\phi\in\cS(\R^2)$ whose Fourier support is contained in a
sufficiently small ball.  Fix $\xi_0\ne0$ and set
\[
f_k(x)=e^{i2^k\xi_0\cdot x}\phi(x),
\]
with the Littlewood--Paley cutoffs chosen so that $\dot\Delta_kf_k=f_k$ for all
large $k$.  Let
\[
k_m=K+mL,
\qquad m=1,\dots,N,
\]
where $L$ is large enough to separate the annuli and $K$ is large enough for
the intervals below to lie in $(0,T)$.  Set
\begin{equation}
\label{eq:counterexample-data}
u_0^N=\sum_{m=1}^Nf_{k_m}.
\end{equation}
Since $|f_{k_m}|=|\phi|$ and the frequency blocks are separated,
\begin{equation}
\label{eq:counterexample-input}
\|u_0^N\|_{\dot F^0_{1,q}}
\sim
N^{1/q}\|\phi\|_{L^1}.
\end{equation}
Moreover, $\|\Delta f_k\|_{L^1}\lesssim2^{2k}$, and the $L^1$-contractivity of
the heat semigroup gives
\[
\|e^{t\Delta}f_k-f_k\|_{L^1}
\le
\int_0^t\|e^{s\Delta}\Delta f_k\|_{L^1}\,ds
\le
t\|\Delta f_k\|_{L^1}.
\]
Thus there are $c_0,c>0$ such that
\begin{equation}
\label{eq:heat-block-lower-bound}
\|e^{t\Delta}f_k\|_{L^1}\ge c
\qquad
(0<t\le c_0 2^{-2k}).
\end{equation}
Define
\[
I_m=
\left[\frac{c_0}{2}2^{-2k_m},c_02^{-2k_m}\right].
\]
After increasing $L$, these intervals are pairwise disjoint.  For $t\in I_m$, \eqref{eq:heat-block-lower-bound} gives
\begin{equation}
\label{eq:counterexample-pointwise-output}
\|e^{t\Delta}u_0^N\|_{\dot F^1_{1,q}}
\ge
2^{k_m}\|e^{t\Delta}f_{k_m}\|_{L^1}
\gtrsim2^{k_m}.
\end{equation}
Because the intervals $I_m$ are pairwise disjoint, integrating
\eqref{eq:counterexample-pointwise-output} yields
\begin{align}
\label{eq:counterexample-output}
\|e^{t\Delta}u_0^N\|_{L^2(0,T;\dot F^1_{1,q})}^2
&\ge
\sum_{m=1}^N\int_{I_m}
\|e^{t\Delta}u_0^N\|_{\dot F^1_{1,q}}^2\,dt\\
&\gtrsim
\sum_{m=1}^N2^{2k_m}|I_m|
\gtrsim N.
\end{align}
Comparing \eqref{eq:counterexample-input} and
\eqref{eq:counterexample-output}, the asserted estimate would imply
$N^{1/2}\lesssim N^{1/q}$ for all $N$, a contradiction when $q>2$.
\end{proof}

\subsection{Proof of the generalized heat estimate}

The positive theory begins with a dyadic kernel estimate that records both
the exponential decay on the parabolic time scale and rapid spatial decay.
This estimate will permit the time norm to be taken at each frequency before
the dyadic summation.

\begin{lemma}[Localized heat kernel]
\label{lem:localized-heat-kernel}
Let
\[
K_j(t,x)
=
\cF^{-1}\left[e^{-t|\xi|^2}\widetilde\varphi(2^{-j}\xi)\right](x),
\]
where $\widetilde\varphi\in C_c^\infty(\R^2\setminus\{0\})$ equals one on the
support of the Littlewood--Paley bump.  For every integer $N\ge0$, there are
constants $C_N,c>0$, independent of $j$ and $t$, such that
\[
|K_j(t,x)|
\le
C_N2^{2j}e^{-ct2^{2j}}(1+2^j|x|)^{-N}.
\]
Consequently, if
\[
L_j(x)=2^{2j}(1+2^j|x|)^{-N},
\]
then for $1\le b\le\infty$,
\[
\|K_j(\cdot,x)\|_{L^b(0,\infty)}
\lesssim
2^{-2j/b}L_j(x).
\]
\end{lemma}

\begin{proof}
Scaling gives
\begin{equation}
\label{eq:localized-kernel-scaling}
K_j(t,x)=2^{2j}K(t2^{2j},2^jx),
\qquad
K(s,y)=\cF^{-1}
\left[e^{-s|\eta|^2}\widetilde\varphi(\eta)\right](y).
\end{equation}
Because $\widetilde\varphi$ is supported in a compact annulus, there exists
$c>0$ such that $|\eta|^2\ge2c$ on its support.  We write
\[
e^{-s|\eta|^2}\widetilde\varphi(\eta)
=
e^{-cs}m_s(\eta),
\qquad
m_s(\eta)=e^{-s(|\eta|^2-c)}\widetilde\varphi(\eta).
\]
For every multi-index $\alpha$, differentiation of $m_s$ produces a finite
sum of terms consisting of a polynomial in $s$ multiplied by
$e^{-s(|\eta|^2-c)}$ and a smooth compactly supported function.  Since
$|\eta|^2-c\ge c$ on the support, the elementary bounds
\[
\sup_{s\ge0}s^m e^{-cs}<\infty
\qquad(m\ge0)
\]
show that $(m_s)_{s\ge0}$ is bounded in $C_c^L$ for every $L$.  Repeated
integration by parts in the inverse Fourier transform therefore yields,
uniformly in $s\ge0$,
\begin{equation}
\label{eq:localized-symbol-decay}
|\cF^{-1}m_s(y)|
\le C_N(1+|y|)^{-N}.
\end{equation}
Consequently,
\begin{equation}
\label{eq:unit-kernel-decay}
|K(s,y)|
\le
C_Ne^{-cs}(1+|y|)^{-N}.
\end{equation}
Combining \eqref{eq:localized-kernel-scaling} and
\eqref{eq:unit-kernel-decay} proves
\begin{equation}
\label{eq:localized-kernel-decay}
|K_j(t,x)|
\le
C_N2^{2j}e^{-ct2^{2j}}(1+2^j|x|)^{-N}.
\end{equation}
For $1\le b<\infty$, \eqref{eq:localized-kernel-decay} gives
\begin{align}
\label{eq:localized-kernel-time-norm}
\|K_j(\cdot,x)\|_{L^b(0,\infty)}
&\lesssim
L_j(x)
\left(\int_0^\infty e^{-cbt2^{2j}}\,dt\right)^{1/b}\\
&\lesssim
2^{-2j/b}L_j(x).
\end{align}
The case $b=\infty$ follows directly from
\eqref{eq:localized-kernel-decay}.
\end{proof}

With the localized kernel bound in hand, we treat the free evolution and the
Duhamel term separately.  In both cases the spatial convolution is reduced to
a Banach-valued Peetre maximal function, while Young's inequality in time
determines the gain of regularity.

\begin{proof}[Proof of Theorem~\ref{thm:general-heat}]
Choose parameters
\[
0<\theta<\min\{1,q\},
\qquad
M>\frac2\theta,
\qquad
N>M+2,
\]
and we split
\begin{equation}
\label{eq:heat-solution-decomposition}
u=u^{\mathrm{lin}}+u^{\mathrm{duh}},
\end{equation}
where
\begin{equation}
\label{eq:heat-linear-duhamel-parts}
u^{\mathrm{lin}}(t)=e^{t\Delta}u_0,
\qquad
u^{\mathrm{duh}}(t)=
\int_0^t e^{(t-\tau)\Delta}F(\tau)\,d\tau.
\end{equation}
For the free heat flow, insert a cutoff equal to one on the support of
$\dot\Delta_j$ and use Lemma~\ref{lem:localized-heat-kernel}:
\[
\dot\Delta_j e^{t\Delta}u_0(x)
=
\int_{\R^2}K_j(t,y)\dot\Delta_ju_0(x-y)\,dy.
\]
Minkowski's inequality in time and
\eqref{eq:localized-kernel-time-norm} give
\begin{align}
\label{eq:free-block-time-bound}
\|\dot\Delta_j e^{t\Delta}u_0(\cdot,x)\|_{L^r_T}
&\le
\int_{\R^2}
\|K_j(\cdot,y)\|_{L^r(0,T)}
|\dot\Delta_ju_0(x-y)|\,dy\\
&\lesssim
2^{-2j/r}L_j*|\dot\Delta_ju_0|(x).
\end{align}
Here and below $2^{-2j/\infty}=1$.  Define
\[
(\dot\Delta_ju_0)^*(x)
=
\sup_y
\frac{|\dot\Delta_ju_0(x-y)|}{(1+2^j|y|)^M}.
\]
Since
\[
|\dot\Delta_ju_0(x-y)|
\le
(\dot\Delta_ju_0)^*(x)(1+2^j|y|)^M,
\]
our choice $N>M+2$ yields
\begin{align}
\label{eq:weighted-convolution-peetre}
L_j*|\dot\Delta_ju_0|(x)
&\le
(\dot\Delta_ju_0)^*(x)
\int_{\R^2}2^{2j}(1+2^j|y|)^{-N+M}\,dy\\
&\lesssim
(\dot\Delta_ju_0)^*(x).
\end{align}
Combining \eqref{eq:free-block-time-bound} and
\eqref{eq:weighted-convolution-peetre}, we obtain
\begin{equation}
\label{eq:free-dyadic-bound}
2^{j(s+\sigma)}
\|\dot\Delta_j e^{t\Delta}u_0(\cdot,x)\|_{L^r_T}
\lesssim
2^{j(s+\sigma-2/r)}(\dot\Delta_ju_0)^*(x).
\end{equation}
The identity
\[
s+\sigma-\frac2r=s+2-\frac2a
\]
then gives
\[
2^{j(s+\sigma)}
\|\dot\Delta_j e^{t\Delta}u_0(\cdot,x)\|_{L^r_T}
\lesssim
2^{j(s+2-2/a)}(\dot\Delta_ju_0)^*(x).
\]
Taking the $L^1_x\ell^q_j$ norm in
\eqref{eq:free-dyadic-bound} and applying
Corollary~\ref{cor:peetre-sequence}, with the harmless dyadic weights included
in the sequence, proves
\begin{equation}
\label{eq:free-heat-estimate}
\|e^{t\Delta}u_0\|_{\dot F^{s+\sigma}_{1,q}(L^r_T)}
\lesssim
\|u_0\|_{\dot F^{s+2-2/a}_{1,q}}.
\end{equation}

We next consider the Duhamel term.  Since $a\le r$, choose
$b\in[1,\infty]$ so that
\begin{equation}
\label{eq:time-young-exponents}
1+\frac1r=\frac1a+\frac1b,
\qquad
\frac1b=1+\frac1r-\frac1a=\frac\sigma2.
\end{equation}
Extend $F$ by zero outside $(0,T)$ and let
\[
\kappa_j(t,y)=\mathbf1_{(0,\infty)}(t)K_j(t,y).
\]
For every fixed $x$,
\begin{equation}
\label{eq:duhamel-time-convolution}
\dot\Delta_j u^{\mathrm{duh}}(t,x)
=
\int_{\R^2}
\bigl[\kappa_j(\cdot,y)*_t
\dot\Delta_jF(\cdot,x-y)\bigr](t)\,dy.
\end{equation}
Minkowski's inequality in $y$, Young's convolution inequality in time, and
\eqref{eq:time-young-exponents} give
\begin{align}
\label{eq:duhamel-block-time-bound}
\|\dot\Delta_j u^{\mathrm{duh}}(\cdot,x)\|_{L^r_T}
&\le
\int_{\R^2}
\|\kappa_j(\cdot,y)*_t
\dot\Delta_jF(\cdot,x-y)\|_{L^r(\R)}\,dy\\
&\le
\int_{\R^2}
\|\kappa_j(\cdot,y)\|_{L^b(\R)}
\|\dot\Delta_jF(\cdot,x-y)\|_{L^a_T}\,dy.
\end{align}
Lemma~\ref{lem:localized-heat-kernel} and
\eqref{eq:time-young-exponents} imply
\begin{equation}
\label{eq:duhamel-kernel-time-bound}
\|\kappa_j(\cdot,y)\|_{L^b(\R)}
\lesssim
2^{-j\sigma}L_j(y).
\end{equation}
Substituting \eqref{eq:duhamel-kernel-time-bound} into
\eqref{eq:duhamel-block-time-bound}, we obtain
\begin{equation}
\label{eq:duhamel-dyadic-bound}
2^{j(s+\sigma)}
\|\dot\Delta_j u^{\mathrm{duh}}(\cdot,x)\|_{L^r_T}
\lesssim
L_j*G_j(x),
\end{equation}
where
\[
G_j(x)=2^{js}\|\dot\Delta_jF(\cdot,x)\|_{L^a_T}.
\]
It is important here that the Banach-valued function
\[
\mathcal G_j(x)=2^{js}\dot\Delta_jF(\cdot,x)
\in E:=L^a(0,T)
\]
is still frequency localized at scale $2^j$ in the spatial variable.  Let
\[
\mathcal G_j^*(x)
=
\sup_y
\frac{\|\mathcal G_j(x-y)\|_E}{(1+2^j|y|)^M}.
\]
Exactly as in \eqref{eq:weighted-convolution-peetre}, $N>M+2$ gives
\begin{equation}
\label{eq:duhamel-peetre-convolution}
L_j*G_j(x)
\lesssim
\mathcal G_j^*(x).
\end{equation}
Taking the mixed norm in \eqref{eq:duhamel-dyadic-bound}, using
\eqref{eq:duhamel-peetre-convolution}, and then applying
Corollary~\ref{cor:peetre-sequence}, we get
\begin{equation}
\label{eq:duhamel-heat-estimate}
\|u^{\mathrm{duh}}\|_{\dot F^{s+\sigma}_{1,q}(L^r_T)}
\lesssim
\|(\mathcal G_j^*)_j\|_{L^1_x\ell^q_j}
\lesssim
\|(\|\mathcal G_j(x)\|_E)_j\|_{L^1_x\ell^q_j}
=
\|F\|_{\dot F^s_{1,q}(L^a_T)}.
\end{equation}
Combining \eqref{eq:heat-solution-decomposition},
\eqref{eq:free-heat-estimate}, and \eqref{eq:duhamel-heat-estimate} proves
the first assertion.

For the Laplacian, frequency localization gives
\begin{equation}
\label{eq:laplacian-block-bound}
\|\dot\Delta_j\Delta u(\cdot,x)\|_{L^r_T}
\lesssim
2^{2j}\|\dot\Delta_j u(\cdot,x)\|_{L^r_T}.
\end{equation}
Multiplication of \eqref{eq:laplacian-block-bound} by
$2^{j(s+\sigma-2)}$ and summation in $L^1_x\ell^q_j$ yield
\begin{equation}
\label{eq:laplacian-full-bound}
\|\Delta u\|_{\dot F^{s+\sigma-2}_{1,q}(L^r_T)}
\lesssim
\|u\|_{\dot F^{s+\sigma}_{1,q}(L^r_T)}.
\end{equation}
Finally, the equation
\[
\partial_tu=\Delta u+F
\]
places $\partial_tu$ in the stated sum space and gives the corresponding norm
bound.  Every estimate above uses the time norm of the kernel on
$(0,\infty)$, so all constants are independent of $T$.
\end{proof}

The diagonal choice of the input and output time exponents places
$\Delta u$ and the forcing term in the same time-refined space. Take $a=r=\rho$ in Theorem~\ref{thm:general-heat}, we get
this direct maximal-regularity consequence as follows. 

\begin{corollary}[Diagonal maximal regularity]
\label{cor:max-reg}
Let $1\le\rho\le\infty$, $1\le q<\infty$, and $s\in\R$.  If
\[
u_0\in\dot F^{s+2-2/\rho}_{1,q},
\qquad
F\in\dot F^s_{1,q}(L^\rho_T),
\]
then
\[
\|u\|_{\dot F^{s+2}_{1,q}(L^\rho_T)}
+
\|\Delta u\|_{\dot F^s_{1,q}(L^\rho_T)}
+
\|\partial_tu\|_{\dot F^s_{1,q}(L^\rho_T)}
\lesssim
\|u_0\|_{\dot F^{s+2-2/\rho}_{1,q}}
+
\|F\|_{\dot F^s_{1,q}(L^\rho_T)}.
\]
\end{corollary}

\section{Auxiliary vector-valued estimates}
\label{sec:auxiliary}

We now develop the auxiliary estimates needed to control the Keller--Segel
drift.  The Bernstein and localized multiplier bounds provide dyadic
derivative estimates in Banach-valued spaces.  A geometric packing argument,
combined with the sequence-valued Peetre estimate, then yields the endpoint
Jawerth-type summability.  The section concludes by applying these tools to
the homogeneous Poisson field.

We start with the Banach-valued form of the Bernstein inequality.

\begin{lemma}[Vector-valued Bernstein inequality]
\label{lem:bernstein}
Let $E$ be a Banach space, let $\alpha$ be a multi-index, and suppose
$\supp\widehat U\subset\{|\xi|\le C2^j\}$.  For
$1\le p\le r\le\infty$,
\[
\|\nabla^\alpha U\|_{L^r_x(E)}
\lesssim
2^{j|\alpha|}2^{2j(1/p-1/r)}\|U\|_{L^p_x(E)}.
\]
If $\widehat U$ is supported in an annulus
$\{c2^j\le|\xi|\le C2^j\}$, then
\[
\|\nabla U\|_{L^p_x(E)}\sim2^j\|U\|_{L^p_x(E)}.
\]
\end{lemma}

\begin{proof}
Choose $\chi\in C_c^\infty(\R^2)$ equal to one on the normalized support of
$\widehat U$ and define $K=\mathcal F^{-1}\bigl[(i\xi)^\alpha\chi(\xi)\bigr]
\in\mathcal S(\mathbb R^2)$.  Then
\begin{equation}
\label{eq:bernstein-kernel-representation}
\nabla^\alpha U
=
2^{j|\alpha|}K_j*U,
\qquad
K_j(x)=2^{2j}K(2^jx).
\end{equation}
Taking the $E$-norm and using
Minkowski's inequality gives
\[
\|\nabla^\alpha U(x)\|_E
\le
2^{j|\alpha|}
(|K_j|*\|U\|_E)(x).
\]
Let $m$ be determined by
\[
1+\frac1r=\frac1m+\frac1p.
\]
Young's inequality applied to
\eqref{eq:bernstein-kernel-representation}, together with the scaling of
$K_j$, implies
\begin{align}
\label{eq:bernstein-direct-proof}
\|\nabla^\alpha U\|_{L^r_x(E)}
&\lesssim
2^{j|\alpha|}\|K_j\|_{L^m}
\|U\|_{L^p_x(E)}\\
&\lesssim
2^{j|\alpha|}2^{2j(1-1/m)}
\|U\|_{L^p_x(E)}\\
&=
2^{j|\alpha|}2^{2j(1/p-1/r)}
\|U\|_{L^p_x(E)}.
\end{align}
This proves the first estimate.

Assume now that $\widehat U$ is supported in an annulus.  Choose
$\chi\in C_c^\infty(\R^2\setminus\{0\})$ equal to one on the normalized
annulus and define
\[
m_\ell(\eta)
=-i\frac{\eta_\ell}{|\eta|^2}\chi(\eta),
\qquad
T_{\ell,j}=m_\ell(2^{-j}D),
\qquad \ell=1,2.
\]
The inverse Fourier transforms of $m_\ell(2^{-j}\cdot)$ are rescaled
Schwartz kernels with uniformly bounded $L^1$ norms.  On the support of $\widehat U$,
\begin{equation}
\label{eq:bernstein-inverse-symbol}
2^{-j}\sum_{\ell=1}^2
m_\ell(2^{-j}\xi)i\xi_\ell
=
\chi(2^{-j}\xi)=1.
\end{equation}
Consequently,
\begin{equation}
\label{eq:bernstein-reconstruction}
U=2^{-j}\sum_{\ell=1}^2T_{\ell,j}\partial_\ell U,
\end{equation}
and the uniform $L^1$ bounds of the kernels in
\eqref{eq:bernstein-reconstruction} imply
\begin{equation}
\label{eq:bernstein-reverse-proof}
\|U\|_{L^p_x(E)}
\lesssim
2^{-j}\|\nabla U\|_{L^p_x(E)}.
\end{equation}
The opposite inequality is the direct Bernstein estimate with
$\alpha$ of order one and $p=r$.  Combining the two inequalities proves the
annular equivalence.
\end{proof}

For the low--high paraproduct, a norm-level Bernstein inequality is not
sufficient: the derivative must be controlled inside the Peetre maximal
function.  The next lemma supplies this localized pointwise version.

\begin{lemma}[Peetre control for localized multipliers]
\label{lem:peetre-multiplier}
Let $E$ be a Banach space, and let $U_j$ be supported in an annulus of size
$2^j$.  If $T_j=m(2^{-j}D)$ with $m\in C_c^\infty$, then
\[
(T_jU_j)^*(x)\lesssim U_j^*(x).
\]
In particular,
\[
(\nabla U_j)^*(x)\lesssim2^jU_j^*(x).
\]
\end{lemma}

\begin{proof}
Let $K=\cF^{-1}m$.  Then
\[
T_jU_j(x)=\int_{\R^2}K_j(y)U_j(x-y)\,dy,
\qquad
K_j(y)=2^{2j}K(2^jy).
\]
For arbitrary $z\in\R^2$, the definition of $U_j^*(x)$ gives
\[
\|U_j(x-z-y)\|_E
\le
U_j^*(x)(1+2^j|z+y|)^M.
\]
The elementary Peetre inequality
\begin{equation}
\label{eq:elementary-peetre-inequality}
1+2^j|z+y|
\le
(1+2^j|z|)(1+2^j|y|)
\end{equation}
therefore implies
\begin{align}
\label{eq:localized-multiplier-peetre}
\frac{\|T_jU_j(x-z)\|_E}{(1+2^j|z|)^M}
&\le
\int_{\R^2}|K_j(y)|
\frac{\|U_j(x-z-y)\|_E}{(1+2^j|z|)^M}\,dy\\
&\le
U_j^*(x)
\int_{\R^2}|K_j(y)|(1+2^j|y|)^M\,dy.
\end{align}
After the change of variables $Y=2^jy$, the last integral becomes
\[
\int_{\R^2}|K(Y)|(1+|Y|)^M\,dY,
\]
which is finite and independent of $j$.  Taking the supremum over $z$ in
\eqref{eq:localized-multiplier-peetre} proves
\begin{equation}
\label{eq:localized-multiplier-final}
(T_jU_j)^*(x)\lesssim U_j^*(x).
\end{equation}
For the derivative assertion, choose a cutoff equal to one on the annular
support of $U_j$ and write
\[
\nabla U_j=2^j\widetilde T_jU_j,
\]
where $\widetilde T_j$ is of the preceding form with a fixed smooth compactly
supported multiplier.  The first estimate then yields
\[
(\nabla U_j)^*(x)\lesssim2^jU_j^*(x).
\]
\end{proof}

The scalar counterpart of the next estimate is a special case of the
classical Jawerth embedding
\[
\dot F^1_{1,q}(\mathbb R^2)
\hookrightarrow
\dot B^0_{2,1}(\mathbb R^2);
\]
see \cite{Jawerth1977,FrazierJawerth1990}.  Since the time-refined
application below requires the temporal norm to remain inside the spatial
Triebel--Lizorkin norm, we give a direct Banach-valued proof.
\begin{lemma}[Banach-valued endpoint Jawerth estimate]
\label{lem:jawerth}Let $1\le q<\infty$, let $E$ be a Banach space, and let
$(U_j)_{j\in\mathbb Z}$ be a family of strongly measurable
$E$-valued functions on $\mathbb R^2$.  Suppose that
\[
\operatorname{supp}\widehat U_j
\subset\{\xi\in\mathbb R^2:|\xi|\le C_0 2^j\},
\qquad j\in\mathbb Z.
\]
Then
\[
\sum_{j\in\Z}2^{-j}\|U_j\|_{L^2_x(E)}
\lesssim
\left\|
\left(\|U_j(x)\|_E\right)_j
\right\|_{L^1_x\ell^q_j}.
\]
In particular,
\[
\sum_{j\in\Z}\|\dot\Delta_jf\|_{L^2_TL^2_x}
\lesssim
\|f\|_{\dot F^1_{1,q}(L^2_T)}.
\]
\end{lemma}

\begin{proof}
We begin with
the dyadic packing estimate used below.  For $j\in\mathbb Z$, let
\[
\mathcal D_j
=
\left\{
2^{-j}\bigl(k+[0,1)^2\bigr):k\in\mathbb Z^2
\right\}
\]
	be the collection of half-open dyadic cubes of side length $2^{-j}$. Let $a_j$ be constant on every dyadic cube
$Q\in\mathcal D_j$, with value $a_Q\ge0$.  Then we claim
\begin{equation}
\label{eq:jawerth-packing-claim}
\sum_j
\left(
\sum_{Q\in\mathcal D_j}(|Q|a_Q)^2
\right)^{1/2}
\lesssim
\left\|
\left(a_j(x)\right)_j
\right\|_{L^1_x\ell^q_j}.
\end{equation}

To prove the claim, set
\[
S(x)=\left(\sum_j a_j(x)^q\right)^{1/q}.
\]
For each fixed $j$, the layer-cake representation and Minkowski's integral
inequality in $\ell^2(\mathcal D_j)$ give
\begin{align}
\label{eq:jawerth-packing-layer-cake}
\left(
\sum_{Q\in\mathcal D_j}(|Q|a_Q)^2
\right)^{1/2}
&\le
\int_0^\infty
\left(
\sum_{\substack{Q\in\mathcal D_j\\a_Q>\lambda}}
|Q|^2
\right)^{1/2}d\lambda.
\end{align}
If $a_Q>\lambda$, then $S(x)>\lambda$ throughout $Q$, and hence
\[
Q\subset\Omega_\lambda:=\{x:S(x)>\lambda\}.
\]
We therefore need the geometric bound
\begin{equation}
\label{eq:jawerth-packing-geometric}
\sum_j
\left(
\sum_{\substack{Q\in\mathcal D_j\\Q\subset\Omega}}|Q|^2
\right)^{1/2}
\lesssim|\Omega|.
\end{equation}
The assertion is trivial when $|\Omega|=\infty$, so assume
$|\Omega|<\infty$.  Decompose all dyadic cubes contained in $\Omega$ into
maximal dyadic cubes $(P_\alpha)_\alpha$.  They are pairwise disjoint, and
every dyadic cube contained in $\Omega$ lies in one of them.  If
$P_\alpha$ has side length $2^{-m_\alpha}$, then for $j\ge m_\alpha$,
\[
\left(
\sum_{\substack{Q\in\mathcal D_j\\Q\subset P_\alpha}}|Q|^2
\right)^{1/2}
=
2^{-(j-m_\alpha)}|P_\alpha|.
\]
Consequently, by Minkowski's inequality in $\ell^2_Q$,
\begin{align}
\sum_j
\left(
\sum_{\substack{Q\in\mathcal D_j\\Q\subset\Omega}}|Q|^2
\right)^{1/2}
&\le
\sum_\alpha\sum_{j\ge m_\alpha}
2^{-(j-m_\alpha)}|P_\alpha|\\
&\lesssim
\sum_\alpha|P_\alpha|
\le|\Omega|,
\end{align}
which proves \eqref{eq:jawerth-packing-geometric}.  Applying this estimate
to $\Omega_\lambda$, summing
\eqref{eq:jawerth-packing-layer-cake} over $j$,
we obtain
\[
\sum_j
\left(
\sum_{Q\in\mathcal D_j}(|Q|a_Q)^2
\right)^{1/2}
\lesssim
\int_0^\infty|\Omega_\lambda|\,d\lambda
=
\int_{\R^2}S(x)\,dx.
\]
This proves the claim.

\medskip
For each $j$ and $Q\in\mathcal D_j$ define
\[
A_Q=\esssup_{z\in Q}\|U_j(z)\|_E.
\]
Let
\[
U_j^*(x)=
\sup_y\frac{\|U_j(x-y)\|_E}{(1+2^j|y|)^M}
\]
with the parameters fixed in Section~\ref{sec:prelim}.  If $x,z\in Q$, then
$2^j|x-z|\le C$, and hence
\[
\|U_j(z)\|_E
\le
(1+2^j|x-z|)^MU_j^*(x)
\lesssim U_j^*(x).
\]
Taking the essential supremum in $z$ gives
\begin{equation}
\label{eq:jawerth-cube-peetre}
A_Q\lesssim\essinf_{x\in Q}U_j^*(x).
\end{equation}
Define the step function
\[
A_j(x)=A_Q
\qquad(x\in Q\in\mathcal D_j).
\]
Then
\[
A_j(x)\lesssim U_j^*(x)
\qquad\text{for almost every }x.
\]

Since $\|U_j(x)\|_E\le A_Q$ on $Q$,
\[
\|U_j\|_{L^2_x(E)}^2
\le
\sum_{Q\in\mathcal D_j}|Q|A_Q^2.
\]
In dimension two, $|Q|=2^{-2j}$, and therefore
\begin{align}
\label{eq:jawerth-block-L2}
2^{-j}\|U_j\|_{L^2_x(E)}
&\le
\left(
\sum_{Q\in\mathcal D_j}(|Q|A_Q)^2
\right)^{1/2}.
\end{align}
Summing \eqref{eq:jawerth-block-L2} over $j$ and applying the claim
\eqref{eq:jawerth-packing-claim}, we get
\begin{equation}
\label{eq:jawerth-packing-step}
\sum_j2^{-j}\|U_j\|_{L^2_x(E)}
\lesssim
\|(A_j)_j\|_{L^1_x\ell^q_j}.
\end{equation}
By \eqref{eq:jawerth-cube-peetre}, $A_j\lesssim U_j^*$ pointwise.
Corollary~\ref{cor:peetre-sequence} therefore implies
\begin{align}
\label{eq:jawerth-peetre-step}
\|(A_j)_j\|_{L^1_x\ell^q_j}
\lesssim
\|(U_j^*)_j\|_{L^1_x\ell^q_j}\lesssim
\|\bigl(\|U_j(x)\|_E\bigr)_j\|_{L^1_x\ell^q_j}.
\end{align}
Combining \eqref{eq:jawerth-packing-step} and
\eqref{eq:jawerth-peetre-step} proves the estimate.

Finally, take $E=L^2(0,T)$ and
\[
U_j=2^j\dot\Delta_jf.
\]
The frequency-support hypothesis is satisfied, and
\[
\sum_j2^{-j}\|2^j\dot\Delta_jf\|_{L^2_x(L^2_T)}
=
\sum_j\|\dot\Delta_jf\|_{L^2_TL^2_x},
\]
whereas the right-hand side is exactly
$\|f\|_{\dot F^1_{1,q}(L^2_T)}$.  This proves the particular assertion.
\end{proof}
We next estimate the nonlocal velocity  $\nabla(-\Delta)^{-1}g$,
it is a Fourier multiplier of order $-1$ and its action on the $j$-th dyadic block produces a factor of order
$2^{-j}$.  Combining this dyadic multiplier estimate with Bernstein's
inequality and Lemma~\ref{lem:jawerth}, we show that the corresponding
dyadic series converges absolutely in the required norm.  

\begin{lemma}[Homogeneous Poisson reconstruction]
\label{lem:poisson}
Let $1\le q<\infty$ and
$g\in\dot F^1_{1,q}(L^2_T)$,
\[
b:=
\sum_{j\in\Z}\nabla(-\Delta)^{-1}\dot\Delta_jg
\]
then
\[
\|b\|_{L^\infty_xL^2_T}
\lesssim
\|g\|_{\dot F^1_{1,q}(L^2_T)}.
\]
\end{lemma}

\begin{proof}
For each $j$, set
\[
b_j=\nabla(-\Delta)^{-1}\dot\Delta_jg.
\]
The multiplier $i\xi/|\xi|^2$ is smooth on the annular support of
$\dot\Delta_jg$ and has order $-1$.  Plancherel's theorem in the spatial variable, followed by integration in
time, gives
\begin{equation}
\label{eq:poisson-block-L2}
\|b_j\|_{L^2_TL^2_x}
\lesssim
2^{-j}\|\dot\Delta_jg\|_{L^2_TL^2_x}.
\end{equation}
Applying the vector-valued Bernstein inequality with $E=L^2(0,T)$,
$p=2$, and $r=\infty$, and then using \eqref{eq:poisson-block-L2}, we obtain
\begin{align}
\label{eq:poisson-block-Linfty}
\|b_j\|_{L^\infty_xL^2_T}
\lesssim
2^j\|b_j\|_{L^2_x(L^2_T)}
\lesssim
\|\dot\Delta_jg\|_{L^2_TL^2_x}.
\end{align}
Summing \eqref{eq:poisson-block-Linfty} over $j$ gives
\begin{equation}
\label{eq:poisson-series-bound}
\sum_j\|b_j\|_{L^\infty_xL^2_T}
\lesssim
\sum_j\|\dot\Delta_jg\|_{L^2_TL^2_x}.
\end{equation}
Lemma~\ref{lem:jawerth} bounds the right-hand side of
\eqref{eq:poisson-series-bound} by
$\|g\|_{\dot F^1_{1,q}(L^2_T)}$.  Thus $\sum_jb_j$ converges absolutely in
the Banach space $L^\infty_xL^2_T$, and its sum $b$ satisfies
\begin{equation}
\label{eq:poisson-final-bound}
\|b\|_{L^\infty_xL^2_T}
\le
\sum_j\|b_j\|_{L^\infty_xL^2_T}
\lesssim
\|g\|_{\dot F^1_{1,q}(L^2_T)}.
\end{equation}
\end{proof}

\section{The Keller--Segel equation}
\label{sec:KS}

We now combine the linear theory of Section~\ref{sec:heat} with the
auxiliary estimates of Section~\ref{sec:auxiliary}.  The first step is the
endpoint bilinear estimate for the drift, obtained by separating the three
Bony interactions.  We then establish the continuity and small-time
properties needed to carry out the fixed-point argument and prove the main
well-posedness theorem.

For $T\in(0,\infty]$, recall
\begin{equation}
\label{eq:solution-space-X}
X_T^q
=
\dot F^0_{1,q}(L^\infty_T)
\cap
\dot F^1_{1,q}(L^2_T),
\end{equation}
with the sum norm.  Define
\begin{equation}
\label{eq:bilinear-duhamel-operator}
\cB(f,g)(t)
=
\int_0^t e^{(t-\tau)\Delta}
\nabla\cdot\left(f\nabla(-\Delta)^{-1}g\right)(\tau)\,d\tau.
\end{equation}

\subsection{Proof of the bilinear estimate}

The critical task is to place the nonlinear drift in
$\dot F^0_{1,q}(L^1_T)$.  After defining the homogeneous velocity through
Lemma~\ref{lem:poisson}, we use Bony's decomposition and treat the
low--high, high--low, and comparable-frequency interactions by different
endpoint mechanisms.

\begin{proof}[Proof of Proposition~\ref{prop:bilinear}]
 For a sufficiently large fixed integer $N$, Bony's decomposition gives
\begin{equation}
\label{eq:bony-decomposition}
fb=T_bf+T_fb+R(f,b),
\end{equation}
where
\begin{equation}
\label{eq:bony-paraproducts}
T_bf=\sum_k\dot S_{k-N}b\,\dot\Delta_kf,
\qquad
T_fb=\sum_k\dot S_{k-N}f\,\dot\Delta_kb,
\end{equation}
and
\begin{equation}
\label{eq:bony-remainder}
R(f,b)=\sum_k\dot\Delta_kf\,\widetilde{\dot\Delta}_kb,
\qquad
\widetilde{\dot\Delta}_kb
=
\sum_{|\ell-k|\le N}\dot\Delta_\ell b.
\end{equation}
We estimate the three interactions separately.

\medskip
\noindent\textbf{Step 1: the low--high interaction.}

Set
\[
u_k=\dot S_{k-N}b\,\dot\Delta_kf,
\qquad
h_k=\nabla\cdot u_k.
\]
For $N$ sufficiently large, the Fourier support of $u_k$, and hence that of
$h_k$, is contained in an annulus of size $2^k$.  Applying Lemma~\ref{lem:almost} with $E=L^1(0,T)$ gives
\begin{equation}
\label{eq:low-high-synthesis}
\|\nabla\cdot T_bf\|_{\dot F^0_{1,q}(L^1_T)}
\lesssim
\|(h_k^*)_k\|_{L^1_x\ell^q_k}.
\end{equation}
Lemma~\ref{lem:peetre-multiplier}, applied to the localized derivative, yields
\begin{equation}
\label{eq:low-high-derivative-peetre}
h_k^*(x)\lesssim2^ku_k^*(x).
\end{equation}
For each $z\in\R^2$, Hölder's inequality in time gives
\begin{align}
\label{eq:low-high-time-holder}
\|u_k(\cdot,z)\|_{L^1_T}
=
\|\dot S_{k-N}b(\cdot,z)\dot\Delta_kf(\cdot,z)\|_{L^1_T}
\lesssim
\|\dot S_{k-N}b(\cdot,z)\|_{L^2_T}
\|\dot\Delta_kf(\cdot,z)\|_{L^2_T}.
\end{align}
The low-frequency operator $\dot S_{k-N}$ has a convolution kernel with
uniformly bounded $L^1$ norm.  Hence, by Minkowski's inequality,
\begin{equation}
\label{eq:low-frequency-velocity-bound}
\|\dot S_{k-N}b(\cdot,z)\|_{L^2_T}
\lesssim
\|b\|_{L^\infty_xL^2_T}.
\end{equation}
Combining \eqref{eq:low-high-time-holder} and
\eqref{eq:low-frequency-velocity-bound}, we obtain
\begin{align}
\label{eq:low-high-ustar}
u_k^*(x)
=
\sup_y
\frac{\|u_k(\cdot,x-y)\|_{L^1_T}}
{(1+2^k|y|)^M}\lesssim
\|b\|_{L^\infty_xL^2_T}
\sup_y
\frac{\|\dot\Delta_kf(\cdot,x-y)\|_{L^2_T}}
{(1+2^k|y|)^M}.
\end{align}
Consequently,
\[
\begin{aligned}
\|\nabla\cdot T_bf\|_{\dot F^0_{1,q}(L^1_T)}
&\lesssim
\|b\|_{L^\infty_xL^2_T}\\
&\quad\times
\left\|
\left(
2^k
\sup_y
\frac{\|\dot\Delta_kf(\cdot,x-y)\|_{L^2_T}}
{(1+2^k|y|)^M}
\right)_k
\right\|_{L^1_x\ell^q_k}.
\end{aligned}
\]
The sequence-valued Peetre estimate, with $E=L^2(0,T)$, gives
\begin{equation}
\label{eq:low-high-sequence-peetre}
\left\|
\left(
2^k
\sup_y
\frac{\|\dot\Delta_kf(\cdot,x-y)\|_{L^2_T}}
{(1+2^k|y|)^M}
\right)_k
\right\|_{L^1_x\ell^q_k}
\lesssim
\|f\|_{\dot F^1_{1,q}(L^2_T)}.
\end{equation}
Combining \eqref{eq:low-high-synthesis}--\eqref{eq:low-high-sequence-peetre}
with the Poisson estimate \eqref{eq:poisson-final-bound}, we conclude that
\begin{equation}
\label{eq:low-high-final}
\|\nabla\cdot T_bf\|_{\dot F^0_{1,q}(L^1_T)}
\lesssim
\|f\|_{\dot F^1_{1,q}(L^2_T)}
\|g\|_{\dot F^1_{1,q}(L^2_T)}.
\end{equation}

\medskip
\noindent\textbf{Step 2: the high--low interaction.}

Since $\ell^1\hookrightarrow\ell^q$ for $q\ge1$,
\begin{equation}
\label{eq:besov-one-to-F}
\dot B^0_{1,1}(L^1_T)
\hookrightarrow
\dot F^0_{1,q}(L^1_T).
\end{equation}
The product $\dot S_{k-N}f\,\dot\Delta_kb$ is spectrally supported in an
annulus of size $2^k$.  Therefore \eqref{eq:besov-one-to-F} and Bernstein's
inequality imply
\begin{equation}
\label{eq:high-low-annular}
\|\nabla\cdot T_fb\|_{\dot F^0_{1,q}(L^1_T)}
\lesssim
\sum_k2^k
\|\dot S_{k-N}f\,\dot\Delta_kb\|_{L^1_TL^1_x}.
\end{equation}
Hölder's inequality in both time and space gives
\begin{equation}
\label{eq:high-low-holder}
2^k
\|\dot S_{k-N}f\,\dot\Delta_kb\|_{L^1_TL^1_x}
\le
\|\dot S_{k-N}f\|_{L^2_TL^2_x}
\|2^k\dot\Delta_kb\|_{L^2_TL^2_x}.
\end{equation}
The order $-1$ of the Poisson multiplier implies
\begin{equation}
\label{eq:high-low-poisson-block}
\|2^k\dot\Delta_kb\|_{L^2_TL^2_x}
\lesssim
\|\dot\Delta_kg\|_{L^2_TL^2_x}.
\end{equation}
Also,
\[
\|\dot S_{k-N}f\|_{L^2_TL^2_x}
\le
\sum_{m<k-N}
\|\dot\Delta_mf\|_{L^2_TL^2_x}.
\]
Substituting \eqref{eq:high-low-holder} and
\eqref{eq:high-low-poisson-block} into \eqref{eq:high-low-annular}, and then
using the low-frequency sum, we obtain
\begin{equation}\label{eq:high-low-sum}
\begin{aligned}
\|\nabla\cdot T_fb\|_{\dot F^0_{1,q}(L^1_T)}
&\lesssim
\sum_k
\left(
\sum_{m<k-N}\|\dot\Delta_mf\|_{L^2_TL^2_x}
\right)
\|\dot\Delta_kg\|_{L^2_TL^2_x}\\
&\le
\left(
\sum_m\|\dot\Delta_mf\|_{L^2_TL^2_x}
\right)
\left(
\sum_k\|\dot\Delta_kg\|_{L^2_TL^2_x}
\right).
\end{aligned}
\end{equation}
Applying Lemma~\ref{lem:jawerth} to both factors in
\eqref{eq:high-low-sum} proves
\begin{equation}
\label{eq:high-low-final}
\|\nabla\cdot T_fb\|_{\dot F^0_{1,q}(L^1_T)}
\lesssim
\|f\|_{\dot F^1_{1,q}(L^2_T)}
\|g\|_{\dot F^1_{1,q}(L^2_T)}.
\end{equation}

\medskip
\noindent\textbf{Step 3: the comparable-frequency remainder.}

Comparable input frequencies may generate output at any lower dyadic
level.  By the Fourier-support properties of the dyadic cutoffs, there
exists a fixed integer $N_1\ge1$, depending only on the chosen
Littlewood--Paley decomposition and on $N$, such that
\[
\dot\Delta_j
\left(
\dot\Delta_kf\,
\widetilde{\dot\Delta}_kb
\right)
=0
\qquad\text{whenever }j>k+N_1.
\]
 Since the derivative falls on the output frequency,
\begin{equation}
\label{eq:remainder-double-sum}
\|\nabla\cdot R(f,b)\|_{\dot B^0_{1,1}(L^1_T)}
\lesssim
\sum_j\sum_{k\ge j-N_1}2^j
\|\dot\Delta_kf\,\widetilde{\dot\Delta}_kb
\|_{L^1_TL^1_x}.
\end{equation}
All summands are nonnegative, so Tonelli's theorem allows us to reverse the
order of summation.  Since
\begin{equation}
\label{eq:remainder-geometric-sum}
\sum_{j\le k+N_1}2^j\lesssim2^k,
\end{equation}
we deduce from \eqref{eq:remainder-double-sum} that
\begin{equation}
\label{eq:remainder-single-sum}
\|\nabla\cdot R(f,b)\|_{\dot B^0_{1,1}(L^1_T)}
\lesssim
\sum_k2^k
\|\dot\Delta_kf\,\widetilde{\dot\Delta}_kb
\|_{L^1_TL^1_x}.
\end{equation}
By Hölder's inequality and the Poisson multiplier estimate,
\begin{align}
\label{eq:remainder-block-holder}
2^k
\|\dot\Delta_kf\,\widetilde{\dot\Delta}_kb
\|_{L^1_TL^1_x}
\le
\|\dot\Delta_kf\|_{L^2_TL^2_x}
\|2^k\widetilde{\dot\Delta}_kb\|_{L^2_TL^2_x}\lesssim
\|\dot\Delta_kf\|_{L^2_TL^2_x}
\|\widetilde{\dot\Delta}_kg\|_{L^2_TL^2_x}.
\end{align}
Combining \eqref{eq:besov-one-to-F},
\eqref{eq:remainder-single-sum}, and \eqref{eq:remainder-block-holder}, and
then applying Lemma~\ref{lem:jawerth}, gives
\begin{align}
\label{eq:remainder-final}
\|\nabla\cdot R(f,b)\|_{\dot F^0_{1,q}(L^1_T)}
\lesssim
\left(
\sum_k\|\dot\Delta_kf\|_{L^2_TL^2_x}
\right)
\left(
\sum_k\|\dot\Delta_kg\|_{L^2_TL^2_x}
\right)\lesssim
\|f\|_{\dot F^1_{1,q}(L^2_T)}
\|g\|_{\dot F^1_{1,q}(L^2_T)}.
\end{align}
Equations \eqref{eq:low-high-final}, \eqref{eq:high-low-final}, and
\eqref{eq:remainder-final} prove the nonlinear estimate.

Finally, Theorem~\ref{thm:general-heat} with
$(a,r,s)=(1,\infty,0)$ and $(1,2,0)$ yields respectively the
$\dot F^0_{1,q}(L^\infty_T)$ and
$\dot F^1_{1,q}(L^2_T)$ estimates for the Duhamel term.  Their sum is the
asserted $X_T^q$ bound.
\end{proof}

\subsection{Strong continuity and small-time smoothing}

The fixed-point construction requires more than uniform linear bounds.  We
also need a strong initial trace, absolute continuity of the time-refined
norm on short intervals, and smallness of the linear smoothing component as
$T\downarrow0$.  The first lemma records the two continuity properties.

\begin{lemma}[Strong continuity and time localization]
\label{lem:continuity-localization}
Let $1\le q<\infty$ and $s\in\R$.
\begin{enumerate}[label=\textup{(\roman*)}]
\item For every $f\in\dot F^s_{1,q}$,
\[
\|e^{t\Delta}f-f\|_{\dot F^s_{1,q}}\longrightarrow0
\qquad\text{as }t\downarrow0.
\]
\item If $1\le\rho<\infty$ and
$u\in\dot F^s_{1,q}(L^\rho(0,T))$, then
\[
\|\mathbf1_Iu\|_{\dot F^s_{1,q}(L^\rho(0,T))}\longrightarrow0
\]
whenever the measurable intervals $I\subset(0,T)$ satisfy $|I|\to0$.
\end{enumerate}
\end{lemma}

\begin{proof}
We first prove (i).  Write $f_j=\dot\Delta_jf$.  For each fixed $j$, the
function $f_j$ is smooth and frequency localized, and hence
\[
e^{t\Delta}f_j(x)\longrightarrow f_j(x)
\]
for every $x$ as $t\downarrow0$.  The localized heat kernel estimate, with $0<t\le1$, gives
\begin{equation}
\label{eq:continuity-heat-majorant}
|e^{t\Delta}f_j(x)|
\lesssim
L_j*|f_j|(x)
\lesssim
f_j^*(x),
\end{equation}
where the last inequality follows by choosing the kernel decay exponent larger
than $M+2$.  Since also $|f_j(x)|\le f_j^*(x)$,
\begin{equation}
\label{eq:continuity-difference-majorant}
2^{js}|\dot\Delta_j(e^{t\Delta}f-f)(x)|
\lesssim
2^{js}f_j^*(x).
\end{equation}
Corollary~\ref{cor:peetre-sequence} shows that
\begin{equation}
\label{eq:continuity-majorant-norm}
\left\|(2^{js}f_j^*)_j
\right\|_{L^1_x\ell^q_j}
\lesssim
\|f\|_{\dot F^s_{1,q}}.
\end{equation}
Dominated convergence, first in the $\ell^q$ variable and then in $x$, now
yields
\[
\|e^{t\Delta}f-f\|_{\dot F^s_{1,q}}\to0.
\]

For (ii), define
\[
A_{j,I}(x)
=
2^{js}\|\dot\Delta_j u(\cdot,x)\|_{L^\rho(I)}.
\]
For every fixed $j$ and almost every $x$, absolute continuity of the
$L^\rho$ integral implies $A_{j,I}(x)\to0$ as $|I|\to0$.  Moreover,
\begin{equation}
\label{eq:time-localization-majorant}
0\le A_{j,I}(x)
\le
2^{js}\|\dot\Delta_j u(\cdot,x)\|_{L^\rho(0,T)}.
\end{equation}
The $\ell^q$-norm of the dominating sequence in
\eqref{eq:time-localization-majorant} is integrable in $x$ by the assumption
on $u$.  Dominated convergence in $L^1_x\ell^q_j$ therefore proves
\[
\|\mathbf1_Iu\|_{\dot F^s_{1,q}(L^\rho(0,T))}\to0.
\]
\end{proof}

Strong continuity alone does not imply that the full
$\dot F^1_{1,q}(L^2_T)$ smoothing norm becomes small.  The following
dyadic dominated-convergence argument provides precisely the smallness
needed for local existence with arbitrary initial data.

\begin{lemma}[Small-time smoothing]
\label{lem:small-time}
If $1\le q<\infty$ and $u_0\in\dot F^0_{1,q}$, then
\[
\|e^{t\Delta}u_0\|_{\dot F^1_{1,q}(L^2(0,T))}
\longrightarrow0
\qquad\text{as }T\downarrow0.
\]
\end{lemma}

\begin{proof}
Let $u_j=\dot\Delta_ju_0$.  From the localized heat kernel estimate and the weighted convolution
argument used in Theorem~\ref{thm:general-heat},
\begin{equation}
\label{eq:small-time-pointwise-heat}
|\dot\Delta_je^{t\Delta}u_0(x)|
\lesssim
e^{-ct2^{2j}}u_j^*(x).
\end{equation}
Consequently,
\begin{equation}\label{eq:small-time-block}
\begin{aligned}
	2^j
	\|\dot\Delta_je^{t\Delta}u_0(\cdot,x)\|_{L^2(0,T)}
	&\lesssim
	2^j
	\left(\int_0^T e^{-2ct2^{2j}}\,dt\right)^{1/2}u_j^*(x)\\
	&\lesssim
	\bigl(1-e^{-2cT2^{2j}}\bigr)^{1/2}u_j^*(x).
\end{aligned}
\end{equation}
For every fixed $j$ and almost every $x$, the coefficient on the right tends
to zero as $T\downarrow0$, and it is bounded uniformly by $1$.  Therefore, for almost every $x$,
\begin{equation}
\label{eq:small-time-pointwise-convergence}
\left(
\sum_j
\left[
2^j\|\dot\Delta_je^{t\Delta}u_0(\cdot,x)\|_{L^2(0,T)}
\right]^q
\right)^{1/q}
\longrightarrow0.
\end{equation}
By \eqref{eq:small-time-block}, the same expression is dominated by
\begin{equation}
\label{eq:small-time-majorant}
C\left(\sum_j(u_j^*(x))^q\right)^{1/q}.
\end{equation}
Corollary~\ref{cor:peetre-sequence} shows that the majorant in
\eqref{eq:small-time-majorant} belongs to $L^1_x$.  Dominated convergence in $x$ now gives
\[
\|e^{t\Delta}u_0\|_{\dot F^1_{1,q}(L^2(0,T))}\to0.
\]
\end{proof}

\subsection{Proof of critical well-posedness}

We now assemble the time-refined heat estimates, the bilinear bound, and the
small-time smoothing property.  The proof first constructs a solution in a
small smoothing ball, then establishes uniqueness in the full class,
continuity in time, continuous dependence, and finally the small-data global
result.

\begin{proof}[Proof of Theorem~\ref{thm:ks-classical}]
Let
\[
A=\|u_0\|_{\dot F^0_{1,q}}.
\]
Theorem~\ref{thm:general-heat} gives a constant $C_L>0$, independent of
$T$, such that
\begin{equation}
\label{eq:fixed-point-linear-bound}
\|e^{t\Delta}u_0\|_{\dot F^0_{1,q}(L^\infty_T)}
+
\|e^{t\Delta}u_0\|_{\dot F^1_{1,q}(L^2_T)}
\le C_LA.
\end{equation}
Proposition~\ref{prop:bilinear} gives a constant $C_B>0$, also independent of
$T$, satisfying
\begin{equation}
\label{eq:fixed-point-bilinear-bound}
\|\cB(f,g)\|_{X_T^q}
\le
C_B
\|f\|_{\dot F^1_{1,q}(L^2_T)}
\|g\|_{\dot F^1_{1,q}(L^2_T)}.
\end{equation}
Choose $\delta>0$ so small that
\begin{equation}
\label{eq:fixed-point-delta-choice}
4C_B\delta\le\frac12,
\qquad
4C_B\delta^2\le\delta.
\end{equation}
By Lemma~\ref{lem:small-time}, there is $T>0$ such that
\begin{equation}
\label{eq:fixed-point-linear-smallness}
\|e^{t\Delta}u_0\|_{\dot F^1_{1,q}(L^2_T)}\le\delta.
\end{equation}
Set $M=C_LA+\delta$ and consider the closed subset
\begin{equation}
\label{eq:fixed-point-ball}
E_T=
\left\{
 u\in X_T^q:\
 \|u\|_{\dot F^0_{1,q}(L^\infty_T)}\le M,
 \quad
 \|u\|_{\dot F^1_{1,q}(L^2_T)}\le2\delta
\right\}.
\end{equation}
Endowed with the metric induced by $X_T^q$, this is a complete metric space.
Define
\[
\Phi(u)=e^{t\Delta}u_0-\cB(u,u).
\]
If $u\in E_T$, then \eqref{eq:fixed-point-linear-bound},
\eqref{eq:fixed-point-bilinear-bound}, and
\eqref{eq:fixed-point-delta-choice} give
\begin{equation}
	\begin{aligned}
		\label{eq:self-map-Linfty}
		\|\Phi(u)\|_{\dot F^0_{1,q}(L^\infty_T)}
		&\le
		C_LA+C_B
		\|u\|_{\dot F^1_{1,q}(L^2_T)}^2\\
		&\le
		C_LA+4C_B\delta^2
		\le M,
	\end{aligned}
\end{equation}
and, using also \eqref{eq:fixed-point-linear-smallness},
\begin{equation}
	\begin{aligned}
		\label{eq:self-map-smoothing}
		\|\Phi(u)\|_{\dot F^1_{1,q}(L^2_T)}
		&\le
		\|e^{t\Delta}u_0\|_{\dot F^1_{1,q}(L^2_T)}
		+C_B\|u\|_{\dot F^1_{1,q}(L^2_T)}^2\\
		&\le
		\delta+4C_B\delta^2
		\le2\delta.
	\end{aligned}
\end{equation}
Equations \eqref{eq:self-map-Linfty} and
\eqref{eq:self-map-smoothing} show that $\Phi(E_T)\subset E_T$.  Moreover, by bilinearity,
\begin{equation}
\label{eq:bilinear-difference-identity}
\cB(u,u)-\cB(v,v)
=
\cB(u-v,u)+\cB(v,u-v).
\end{equation}
For $u,v\in E_T$, \eqref{eq:fixed-point-bilinear-bound},
\eqref{eq:bilinear-difference-identity}, and
\eqref{eq:fixed-point-delta-choice} yield
\begin{equation}
	\begin{aligned}
		\label{eq:local-contraction}
		\|\Phi(u)-\Phi(v)\|_{X_T^q}
		&\le
		C_B
		\left(
		\|u\|_{\dot F^1_{1,q}(L^2_T)}
		+
		\|v\|_{\dot F^1_{1,q}(L^2_T)}
		\right)
		\|u-v\|_{X_T^q}\\
		&\le
		4C_B\delta\|u-v\|_{X_T^q}\\
		&\le
		\frac12\|u-v\|_{X_T^q}.
	\end{aligned}
\end{equation}
Banach's fixed-point theorem therefore gives a unique fixed point of $\Phi$ in
$E_T$, which is a local mild solution.

We now prove uniqueness in the whole class $X_T^q$.  Let $u,v\in X_T^q$ be
mild solutions with the same initial data.  By the time-localization part of
Lemma~\ref{lem:continuity-localization}, we can choose a finite partition
\[
0=t_0<t_1<\cdots<t_N=T
\]
such that, for every $I_m=(t_{m-1},t_m)$,
\begin{equation}
\label{eq:uniqueness-partition-smallness}
C_B\left(
\|u\|_{\dot F^1_{1,q}(L^2(I_m))}
+
\|v\|_{\dot F^1_{1,q}(L^2(I_m))}
\right)<\frac12.
\end{equation}
On the first interval, the difference satisfies
\begin{equation}
\label{eq:uniqueness-mild-difference}
u(t)-v(t)
=
-\int_0^t e^{(t-\tau)\Delta}
\left[
\mathcal N(u)(\tau)-\mathcal N(v)(\tau)
\right]d\tau.
\end{equation}
where
\[
\mathcal N(w)=\nabla\cdot
\left(w\nabla(-\Delta)^{-1}w\right).
\]
Applying the bilinear difference estimate to
\eqref{eq:uniqueness-mild-difference} and using
\eqref{eq:uniqueness-partition-smallness}, we obtain
\begin{equation}
\label{eq:uniqueness-absorption}
\|u-v\|_{X_{I_1}^q}
\le
\frac12\|u-v\|_{X_{I_1}^q}.
\end{equation}
Hence $u=v$ on $I_1$.  Suppose inductively that equality has been proved up to
$t_{m-1}$.  Writing both solutions in shifted mild form from $t_{m-1}$, the
initial difference on $I_m$ vanishes, and the same estimate gives equality on
$I_m$.  Induction proves $u=v$ on $[0,T]$.

We next establish time continuity.  Write
\[
u(t)=e^{t\Delta}u_0-W(t),
\qquad
W(t)=\int_0^t e^{(t-\tau)\Delta}F(\tau)\,d\tau,
\]
where, by Proposition~\ref{prop:bilinear},
\begin{equation}
\label{eq:continuity-forcing}
F=\nabla\cdot
\left(u\nabla(-\Delta)^{-1}u\right)
\in\dot F^0_{1,q}(L^1_T).
\end{equation}
The free heat flow belongs to
$C([0,T];\dot F^0_{1,q})$ by
Lemma~\ref{lem:continuity-localization}.  It remains to treat $W$.

We use an approximation argument that is compatible with the time-refined
norm.  Since $q<\infty$ and the time exponent is one, finite dyadic
truncations are dense in $\dot F^0_{1,q}(L^1_T)$; after fixing finitely many
frequencies, simple functions in time with smooth spatial values are dense in
the corresponding finite sum.  We may therefore choose $F^{(n)}$ such that
\begin{equation}
\label{eq:forcing-approximation}
\|F^{(n)}-F\|_{\dot F^0_{1,q}(L^1_T)}\longrightarrow0,
\end{equation}
where each $F^{(n)}$ has only finitely many dyadic components and is simple in
time.  Define
\[
W_n(t)=\int_0^t e^{(t-\tau)\Delta}F^{(n)}(\tau)\,d\tau.
\]
For such an approximant, every dyadic component is a finite sum of ordinary
Bochner integrals of smooth frequency-localized functions.  The strong
continuity of the heat semigroup and the absolute continuity of the time
integral imply
\[
W_n\in C([0,T];\dot F^0_{1,q}).
\]
On the other hand, Theorem~\ref{thm:general-heat} with
$(a,r,s)=(1,\infty,0)$ gives
\begin{equation}
\label{eq:duhamel-approximation-bound}
\|W_n-W\|_{\dot F^0_{1,q}(L^\infty_T)}
\lesssim
\|F^{(n)}-F\|_{\dot F^0_{1,q}(L^1_T)}.
\end{equation}
Moreover,
\begin{equation}
\label{eq:duhamel-uniform-bound}
\sup_{0\le t\le T}
\|W_n(t)-W(t)\|_{\dot F^0_{1,q}}
\le
\|W_n-W\|_{\dot F^0_{1,q}(L^\infty_T)},
\end{equation}
because the supremum in time can be moved inside each dyadic component to
obtain the time-refined norm.  Thus $W_n\to W$ uniformly in
$C([0,T];\dot F^0_{1,q})$.  Being a uniform limit of continuous functions,
$W$ is continuous.  Consequently,
\[
u\in C([0,T];\dot F^0_{1,q}).
\]

For continuous dependence, suppose
$u_0^{(n)}\to u_0$ in $\dot F^0_{1,q}$.  Choose $T$ such that
\begin{equation}
\label{eq:continuous-dependence-common-time}
\|e^{t\Delta}u_0\|_{\dot F^1_{1,q}(L^2_T)}<\frac\delta2.
\end{equation}
The linear estimate implies
\begin{equation}
\label{eq:continuous-dependence-linear}
\|e^{t\Delta}(u_0^{(n)}-u_0)\|_{\dot F^1_{1,q}(L^2_T)}
\lesssim
\|u_0^{(n)}-u_0\|_{\dot F^0_{1,q}},
\end{equation}
so the same $T$ works for all sufficiently large $n$, with the corresponding
solutions lying in a common contraction ball.  Their difference satisfies
\begin{equation}
	\begin{aligned}
		\label{eq:continuous-dependence-estimate}
		\|u^{(n)}-u\|_{X_T^q}
		&\le
		C_L\|u_0^{(n)}-u_0\|_{\dot F^0_{1,q}}\\
		&\quad+
		C_B\left(
		\|u^{(n)}\|_{\dot F^1_{1,q}(L^2_T)}
		+
		\|u\|_{\dot F^1_{1,q}(L^2_T)}
		\right)
		\|u^{(n)}-u\|_{X_T^q}.
	\end{aligned}
\end{equation}

By the common smallness ensured by
\eqref{eq:continuous-dependence-common-time} and
\eqref{eq:continuous-dependence-linear}, the coefficient of the last term in
\eqref{eq:continuous-dependence-estimate} is strictly smaller than one and
can be absorbed.  Therefore $u^{(n)}\to u$ in $X_T^q$.

Finally, assume that $A$ is small and work on $(0,\infty)$.  Consider
\begin{equation}
\label{eq:global-fixed-point-ball}
E_\infty=
\{u\in X_\infty^q:\|u\|_{X_\infty^q}\le2C_LA\}.
\end{equation}
For $u\in E_\infty$,
\begin{equation}
\label{eq:global-self-map}
\|\Phi(u)\|_{X_\infty^q}
\le
C_LA+C_B(2C_LA)^2
=
C_LA+4C_BC_L^2A^2.
\end{equation}
Thus \eqref{eq:global-self-map} shows that $\Phi$ maps $E_\infty$ into itself
if $4C_BC_LA\le1$.  For $u,v\in E_\infty$,
\begin{equation}
\label{eq:global-contraction}
\|\Phi(u)-\Phi(v)\|_{X_\infty^q}
\le
4C_BC_LA\|u-v\|_{X_\infty^q}.
\end{equation}
If, for example,
\begin{equation}
\label{eq:global-smallness-threshold}
A\le(8C_BC_L)^{-1},
\end{equation}
then $\Phi$ is a contraction on $E_\infty$.  This proves global existence and
uniqueness for sufficiently small initial data.
\end{proof}

\section{The drift--diffusion system}
\label{sec:drift}

The final application shows that the preceding framework is stable under a
simple coupling and a nonnegative damping term.  Since both nonlinearities
have the bilinear form covered by Proposition~\ref{prop:bilinear}, the
argument reduces to a product-space version of the Keller--Segel fixed point.

\begin{proof}[Proof of Corollary~\ref{cor:drift-diffusion}]
The mild formulation of system~\eqref{eq:drift-system-intro} is
\begin{align}
\label{eq:drift-mild-system}
v(t)
&=e^{t\Delta}v_0
-
\int_0^t e^{(t-\tau)\Delta}
\nabla\cdot
\left(w\nabla(-\Delta)^{-1}w\right)(\tau)\,d\tau,\\
w(t)
&=e^{t(\Delta-\lambda)}w_0
-
\int_0^t e^{(t-\tau)(\Delta-\lambda)}
\nabla\cdot
\left(v\nabla(-\Delta)^{-1}w\right)(\tau)\,d\tau.
\end{align}
Since
\begin{equation}
\label{eq:damped-semigroup-factorization}
e^{t(\Delta-\lambda)}=e^{-\lambda t}e^{t\Delta},
\qquad \lambda\ge0,
\end{equation}
the localized kernel bounds, the time-refined heat estimates, and the
small-time smoothing estimate remain valid for this damped semigroup, with
constants no larger than those for $e^{t\Delta}$.

Let
\begin{equation}
\label{eq:product-solution-space}
\mathbf X_T^q=X_T^q\times X_T^q
\end{equation}
with the sum norm and define the fixed-point map
$\mathbf\Phi=(\Phi_1,\Phi_2)$ by the two right-hand sides in
\eqref{eq:drift-mild-system}.  Proposition~\ref{prop:bilinear} gives
\begin{align}
\label{eq:drift-fixed-point-bounds}
\|\Phi_1(v,w)\|_{X_T^q}
&\le
C_L\|v_0\|_{\dot F^0_{1,q}}
+C_B\|w\|_{\dot F^1_{1,q}(L^2_T)}^2,\\
\|\Phi_2(v,w)\|_{X_T^q}
&\le
C_L\|w_0\|_{\dot F^0_{1,q}}
+C_B
\|v\|_{\dot F^1_{1,q}(L^2_T)}
\|w\|_{\dot F^1_{1,q}(L^2_T)}.
\end{align}
The corresponding difference estimates are bilinear and have the same form.
By Lemma~\ref{lem:small-time}, choose $T$ so that both free heat flows are
small in their $\dot F^1_{1,q}(L^2_T)$ components.  Choosing a product ball
with bounded $L^\infty_T\dot F^0_{1,q}$ components and sufficiently small
smoothing components, the preceding estimates show that $\mathbf\Phi$ maps
the ball into itself and is a contraction.  This gives the local mild
solution.

Uniqueness in the full class follows by partitioning the time interval exactly
as in the proof of Theorem~\ref{thm:ks-classical}; on each subinterval, the sum
of the four smoothing norms can be made small enough to absorb the bilinear
difference terms.  Strong continuity and continuous dependence follow from
the same semigroup and time-localization arguments.

If
\begin{equation}
\label{eq:drift-small-data-condition}
\|v_0\|_{\dot F^0_{1,q}}
+
\|w_0\|_{\dot F^0_{1,q}}
\end{equation}
is sufficiently small, the product contraction can instead be performed on
$\mathbf X_\infty^q$.  This proves the global assertion.
\end{proof}
\section*{Acknowledgments}
This work was supported by the Hubei Provincial Natural Science
Foundation of China (Grant No.~2026AFA034) and the National Natural Science
Foundation of China (Grant No.~12531017).

\end{document}